\documentclass[10pt, toc_adj]{amsart}
\usepackage[utf8]{inputenc}
\usepackage[dvips,final]{graphicx}
\usepackage[dvips]{geometry}
\usepackage{color} %include even if images aren’t in color
\usepackage{epsfig}
\usepackage{latexsym}
\usepackage{pstricks}
\usepackage{amssymb, amsmath, amsthm}

\usepackage{epsf}       
\usepackage{psfrag}
\usepackage{auto-pst-pdf}
\usepackage[all]{xy}
\usepackage{enumerate}
\newtheorem{theorem}{Theorem}
\newtheorem*{thm}{Theorem} 
\newtheorem*{pro}{Proposition}
\newtheorem*{lem}{Lemma}
\newtheorem*{cor}{Corollary}
 
\theoremstyle{definition}
\newtheorem*{dfn}{Definition}
\newtheorem*{dfna}{Alternative Definition}
\newtheorem*{exa}{Example}
\newtheorem*{note}{Note}

\numberwithin{figure}{section}

\newcommand{\upperRomannumeral}[1]{\uppercase\expandafter{\romannumeral#1}}
\newcommand{\Spin}{{\operatorname{Spin}}}
\newcommand{\nc}{\newcommand}    
\newcommand{\psdiag}[3]{\hspace{1mm}\raisebox{-#1mm}{\epsfysize#2mm
\epsffile{#3.eps}}\hspace{1mm}}
\def\be#1\ee{\begin{equation}#1\end{equation}}
\nc{\bc}{\begin{center}}
\nc{\ec}{\end{center}}
\nc{\bb}{\mathbb}
\nc{\Z}{{\mathbb Z}}
\nc{\Q}{{\mathbb Q}}
\nc{\N}{{\mathbb N}}
\nc{\Ob}{\operatorname{Ob}}
\nc{\Obj}{\operatorname{Obj}}
\nc{\Rib}{\operatorname{Rib}}
\nc{\Hom}{\operatorname{Hom}}
\nc{\SO}{\operatorname{SO}}
\nc{\BSO}{\operatorname{BSO}}
\nc{\BSpin}{\operatorname{BSpin}}
\nc{\ESO}{\operatorname{ESO}}
\nc{\K}{\operatorname{K}}
\nc{\PK}{\operatorname{PK}}
\nc{\T}{\operatorname{T}}
\nc{\diag}{\operatorname{diag}}
\nc{\Rep}{\operatorname{Rep}}
\nc{\Tors}{\operatorname{Tors}}
\nc{\TorsH}{\Tors H}
\nc\congto{\overset{\cong}{\to}}
\nc\Span{\operatorname{Span}}
\nc\SsS{(\Sigma,s_{\Sigma })}
\nc\G{{\mathcal{G}}}
\nc\id{{\operatorname{id}}}
\nc\Ker{\operatorname{Ker}}
\nc\Homeo{\operatorname{Homeo}}
\nc\MCG{\operatorname{MCG}}
\nc\Id{\mathrm{Id}}
\nc\MOO{\mathrm{MOO}}

\nc{\End}{\operatorname{End}}
\nc{\tr}{\operatorname{tr}}
\nc{\lk}{\operatorname{lk}}
\renewcommand{\Im}{{\operatorname{Im}}}
\nc{\U}{\operatorname{U}}

\nc{\TX}{\operatorname{TX}}
\nc{\TM}{\operatorname{TM}}
\nc{\SU}{\operatorname{SU}}
\nc{\card}{\operatorname{card}}
\nc{\fk}{\mathbf{k}}
\theoremstyle{remark}

\def\1{\mathbf{1}}

\newcommand{\Vect}{\operatorname{Vect}}
\newcommand{\I}{\operatorname{I}}
\newcommand{\fsl}{{\mathfrak{sl}}}
\def\R{\mathbb R}

\def\C{{\mathcal C}}
\def\v8{\vskip 8pt}
\def\a{\alpha}
\def\b{\beta}
\def\la{\langle}
\def\ra{\rangle}
\def\l{\lambda}
\def\n{\nu}
\def\m{\mu}
\def\d{\delta}

\def\k{\mathbf k}
\def\G{\Gamma}
\input epsf
\begin{document}
\date{\today}
\title{Spin modular categories}
\author[Anna Beliakova]{Anna Beliakova}
\author[Christian Blanchet]{Christian Blanchet}
\author[Eva Contreras]{Eva Contreras}

\maketitle
\begin{abstract}
Modular categories are a well-known source of
quantum 3--manifold invariants.  
In this paper we study structures on modular categories which allow to define 
refinements of quantum 3--manifold invariants involving cohomology classes or
generalized spin 
and complex spin structures.
 A crucial role in our
construction is played by objects which are
invertible under tensor product. All known examples of cohomological or spin type refinements
of the Witten-Reshetikhin-Turaev 3--manifold invariants are
 special cases of our construction. In addition,
we establish a splitting formula for the refined invariants, generalizing
the well-known product decomposition of quantum invariants into projective ones and those determined by the linking matrix.
%We give examples of spin $d$-graded modular categories in type A.
 \end{abstract}

%\tableofcontents
\section{Introduction}
In the late 80's, inspired by Witten's ideas \cite{Witten}, Reshetikhin and Turaev \cite{ReshetikhinTuraev}
came up with a construction of new 3--manifold invariants, known as WRT quantum invariants. 
Few years later, Turaev \cite{Tbook} formalized this construction by introducing the notion of
{\it modular category}. His main result is that any modular category gives rise to a quantum $3$--manifold invariant.
%Let us recall the main properties of this important algebraic object.

A modular category is a special kind of ribbon category which has a finite set of
simple objects $\G$, including the unit object $\1\in \G$, satisfying duality, domination and
non-degeneracy axioms.  A ribbon category is a monoidal category with braiding, twist and compatible duality.
 Ribbon categories are  universal receivers for  invariants of ribbon graphs \cite{Tbook}.
Examples of modular categories arise from  representation theory of  quantum groups, when the
quantum parameter $q$ is a root of unity, 
or  can be constructed  skein theoretically \cite{BHMV,L1,Blanchet,BB}.
Many authors  \cite{KM},
\cite{Tu1}, \cite{Bl1} observed independently that for some special values of $q$, the
$\fsl_2$ WRT invariants admit spin and cohomological refinements. Cohomological refinements give rise to
homotopy quantum field theories (HQFT's), constructed by Turaev in \cite{THQFT}. 
%Here the grading group is abelian 
%while Turaev's theory applies for general group.
However, spin refinements do not fit in the framework of HQFT's.

The main aim of this paper is to provide  an algebraic setting for spin type  refinements of quantum 3--manifold invariants.
Before explaining our results let us recall few definitions. 

%Let $H$ be a subgroup of $G$.
% Let us recall the definition 

Given a group $G$,
 a $G$--category was defined in \cite[Section VI]{THQFT} as an additive monoidal category $\C$
with left duality and unit object $\1$ that splits 
as a disjoint union of full subcategories $\{\C_\a\}_{\a\in G}$ such that
\begin{itemize}
%\item
%each $\C_\a$ is a full subcategory of $\C$;
%\item
%each object of $\C$ belongs to $\C_\a$ for a unique $\a$;
\item
$\Hom_\C(\C_\a,\C_\beta)=0$ if $\a\neq \beta$;
\item
for $U\in \C_\a$ and $V\in \C_\beta$, $U \otimes V \in \C_{\a\beta}$; 
\item
$\1\in \C_1$, and for $U\in \C_\a$, $U^*\in \C_{\a^{-1}}$.
\end{itemize}
%\end{definition}

We call an object $t$ of a modular category {\it invertible} if there exists an object $\l$, such that
$t\otimes \lambda \cong \1$. 
%Here we identify objects with their isomorphism classes.
It is easy to see that isomorphism classes of invertible objects form a finite abelian group
under tensor multiplication. Let us denote  by $G=G_\mathcal{C}$
the group of isomorphism classes of invertible objects in the modular category $\mathcal{C}$. 
In Section 5 we  show that
the braiding (or monodromy)
coefficients of $\lambda \in \G$ with $t\in G$ define a map $\G\to \Hom(G, S^1)=\widehat G$.
This map induces  on $\C$ the structure of a  $\widehat G$--category.

Note 
 the braiding matrix defines a bilinear form on $G$, and the twist coefficients extend it to a quadratic form. 
A special role in our approach will play a subgroup $H$ of $G$, such that
 the bilinear form restricted to $H$
is trivial  while  the quadratic extension is not.

% Let $H$ be a subgroup of $G$.
%The braiding (or monodromy)
%coefficients of $\lambda \in \G$ with $t\in H$ define a map $\G\to \Hom(H, S^1)=\widehat H$, which can be 
%interpreted 
%as a $\widehat H$--grading on $\G$. If we allow only direct sums of homogenous objects, the modular category itself 
% is $\widehat H$--graded.
% We denote by $\Gamma_1$ the  part of $\Gamma$ corresponding to the trivial character with respect to this grading.
Let us state  our main definition.

\begin{dfn}\label{def1}
Let $\mathcal{C}$ be a modular category with a group $G$ of invertible objects. 
For any subgroup $H\subset G$, we call  $\mathcal{C}$ 
$H$--refinable if $H \subset \Ob(\C_1)$. Moreover, an $H$--refinable modular category $\C$ 
is   called {\it  $H$--spin} if the twist quadratic form restricted to
$H$ is non-trivial, or equivalently if $H$ has at least one element with twist coefficient $-1$. 
When $H$ is cyclic of order $d$, we will use shorthand 
$d$--spin and $d$--refinable.
\end{dfn}

For example, the $\fsl_2$  modular category at the $r$th root of unity $q$ is $2$--refinable
for $r=0\pmod 2$ and it is  $2$--spin
if $r= 0\pmod 4$. The group $G=H=\Z/2\Z$ is
generated by the $(r-1)$--dimensional representation.

We say that $t\in G$ has order $n$ if $n$ is the minimal integer such that $t^n\cong 1$. 
We will see that the order of an element with twist $-1$
has to be even.
%Moreover, any $H$--spin modular category is a special case of
%an $H$--refinable one.

%\begin{definition}\label{def2}
%A modular category $\mathcal{C}$ 
%is called {\it  $H$--refinable} if
%$H\subset G_{\mathcal{C}}$ is  a  group of invertible objects
% in trivial $\widehat H$ degree, i.e. on which the $S$-matrix bilinear form is trivial. It is called {\it reduced} if it is $G_\mathcal{C}$--refinable.
% \end{definition}

%\AB{Is modular $G$-category like in turaev's book better??}

%Note that any modular category which has an invertible object $t$ in degree zero with twist $-1$  is $H$--spin, where $H$ is 
%generated by $t$.  

%\AB{perhaps better an sl(n) example?}

$H$--refinable modular categories which are non spin give rise
 to invariants
of pairs $(M, h)$ for any  compact orientable 3--manifold $M$ and
a cohomology class $h\in H^1(M;\widehat{H})$. They also fit
in the setting of   modular  group-categories
 introduced in \cite{THQFT}. 
In the spin case, the formalism of \cite{THQFT} does not strictly apply  since the subcategory in trivial degree is not modular. 
%The construction of HQFT with weakened non degeneracy condition will be done elsewhere.

Let us concentrate on the spin case.
%For any  even integer $d$, we define the set $\Spin_d(M)$
% of spin structures modulo $d$, and the set  $\Spin^c_d(M)$
% of  complex spin structures modulo $d$ on 
%a compact orientable 3-manifold 
%$M$.
%If $M$ is obtained by surgery on a link $L$ with linking matrix $(L_{ij})$ 
%these sets have nice combinatorial presentations.
%The second author proved in \cite{Bl1} that elements of $\Spin_d(M)$ are in one-to-one correspondence with
% $$\{(s_1,\dots, s_n)\in (\Z/d\Z)^n \;|\;
% \sum^n_{j=1}L_{ij} s_j=\frac{d}{2} L_{ii} \pmod d \}.$$
%In this paper we will define modulo $d$ complex spin structures and following \cite{Massuyeau2} for the usual complex spin structures we show that they 
%can be described as follows:
%\begin{proposition} For any even integer $d$,
%%$\Spin^c_d$--structures 
%modulo $d$ complex spin structures 
%on the $3$-manifold $M$ are in one-to-one correspondence with the elements of the set
%$$\frac{\{(\sigma_1,\dots,\sigma_n)\in 
%(\Z/2d\Z)^n\;|\; \sigma_i=L_{ii}\pmod 2 \}}{ 2 \Im L},$$
%called modulo $d$ Chern vectors.
%\end{proposition}
If we have a $H$--spin modular category, then the twist coefficients define an order $2$ element $v\in \widehat H$, which 
we will call the spin character.  We will extend the definition of generalized spin 
structures with modulo $d$ coefficients given in \cite{Bl1} 
to this situation and define $(\widehat{H},v)$ generalized spin structures. One of the  results of this paper is the following.

%With any $H$--spin modular category, we associate 
%a spin type structure $\sigma_H$ on $M$
%as  follows. Assume $H=\oplus^k_{i=1} \Z/n_i\Z$,  and
%the generators of the first $l$ cyclic subgroups have twist coefficients $-1$, hence
%$n_i=0\pmod2$ for $1\leq i \leq l$.
%Then
%$$\sigma_H=(\sigma^{(1)},\dots, \sigma^{(l)}, h^{(l+1)},\dots, h^{(k)} )$$
%where $\sigma^{(i)} \in \Spin_{n_i}(M)$ or, if $n_i=0\pmod 4$, 
%we can choose
%$\sigma^{(i)}\in \Spin^c_{n_i/2}(M)$  for $1\leq i\leq l$ and
%$h^{(i)}\in H^1(M;\Z/n_i\Z)$ or, 
%alternatively, $h^{(i)}\in H_1(M;\Z/n_i\Z)$ for  $l<i\leq k$.  
%Note that we can choose $\sigma^{(i)}\in %\Spin^c_{n_i/2}(M)$ only if

\begin{theorem} \label{main}
Any $H$--spin  modular category $\C$ with associated spin character $v\in \mathrm{Tor}_2(\widehat{H})$ 
provides a topological invariant $\tau_\C(M,\sigma)$ of a pair
$(M,\sigma)$, where $\sigma$ is a $(\widehat{H},v)$ generalized spin structure on $M$. Moreover,
$$\tau_\C (M)=\sum_{\sigma} \tau_\C(M,\sigma).$$
\end{theorem}

We expect that Theorem \ref{main} extends naturally to a spin type TQFT.

%\begin{co}
%I think we should also put here the theorem for complex spin invariants 6.3. May be we should have a separate section for complex spin invariants; those are completely new.
%\end{co}
In Section 2 we  define $d$--{\it complex} spin structures. Let us denote the set of such structures on $M$
by $\Spin^c_d(M)$. 
Generalizing results of \cite{Massuyeau2},  we identify $\Spin^c_d(M)$
with the set of modulo $d$ Chern vectors which are further used for 
constructing extensions of WRT invariants.
% for $3$--manifolds with modulo $d$ complex spin structure.

\begin{theorem} \label{main_c} Suppose $d$ is an even positive integer.
For any $2d$--spin  modular category $\C$, there exists
 a topological invariant $\tau_\C(M,\sigma)$ of a pair
$(M,\sigma)$, where $\sigma\in \Spin^c_d(M)$ is a $d$--complex spin structure on $M$.
\end{theorem}

%
%\begin{proposition} For any even integer $d$,
%%$\Spin^c_d$--structures 
%modulo $d$ complex spin structures 
%on the $3$-manifold $M$ are in one-to-one correspondence with the elements of the set
%$$\frac{\{(\sigma_1,\dots,\sigma_n)\in 
%(\Z/2d\Z)^n\;|\; \sigma_i=L_{ii}\pmod 2 \}}{ 2 \Im L},$$
%called modulo $d$ Chern vectors.
%\end{proposition}

We call a modular category $\C$ {\it reduced} if it is $G_\C$--refinable.

Assume  the group of
invertible objects $G=\la t\ra$ of our modular category $\C$ is cyclic, 
but $t\in \Ob(\C_\delta)$ with
$\delta\neq 1$. If
$|G|=d$, then there is a positive integer $m$, such that  $\delta m=d$.
Clearly, $\C$ is not reduced, but
$H$--refinable, where the subgroup $H$ of order $\delta$ is generated by $ t^m\in \C_1$. 

Given these data, one way to
 construct 
 refined invariants is by using the $\delta$--refinable structure on $\C$.
 In Section 7 we show  that there is a more efficient way to compute this invariant.
 Namely, there always exists a smaller
 reduced category $\tilde \C$, which leads to the same invariant up to a 
correction term fully determined by the linking matrix. 
If $\gcd(m,\delta)=1$,  $\tilde\C$ is particularly simple and coincides with $\C_1$ where $1\in \widehat H$.
%In general case, the construction of $\tilde \C$ is described in Section 7.
%Our last theorem relates these two refined invariants.
%In addition, we study the case of a modular category $\mathcal{C}$ with non trivial  braiding bilinear form on the group $G$ of its invertible objects.
%We consider the case where $G$ is cyclic. 
%Let $t$ be a generator of $G$, whose

% We first state the decomposition theorem in the case when
% $\delta|d$ and 
%\mbox{$\gcd(\delta,d/\delta)=1$}.

\begin{theorem}\label{decomp}
 Let $\C$ be a  modular category with a cyclic group of invertible objects $G=\la t\ra$.
Assume $d=|G|$, $t\in \C_\delta$ and $\delta m=d$. 
Then, there exists a reduced $\delta$--refinable category   
 $\widetilde{\C}$, a positive integer $\alpha$ and a root of unity $\xi$  such that
for any closed oriented 3--manifold $M$ we have
 $$\tau_\C(M, \sigma)= \alpha^{-b_1(M)}
\tau_{\widetilde{\C}}(M, \sigma) \tau^{\rm MMO}_{\xi}(M, \sigma)$$
 where 
$\tau^{\rm MMO}_{\xi}(M,\sigma)$ is the refined
   Murakami--Ohtsuki--Okada invariant.
 We have   either  $\sigma\in H^1(M,\Z_\delta)$
or $\sigma\in \Spin_\delta(M)$, and $b_1(M)$ is the first Betti number.

In the particular case, when $gcd(\delta,m)=1$, we have $\alpha=1$, $\xi$ is a root of unity of order $m$ and $\tau^{\rm MMO}_{\xi}(M, \sigma)$ does not depend on $\sigma$ so that we have
$$\tau_\C(M)= 
\tau_{\tilde\C}(M) \tau^{\rm MMO}_{\xi}(M)\ . $$
%and $\widetilde{\C}=\C_0$ is a full--subcategory of $\C$ with
%with $G_{\widetilde{\C}}=\la t^m\ra$.

%If $\delta$ is even, then $\widetilde \C$ is $\delta$--refinable.
\end{theorem}

The Murakami--Ohtsuki--Okada invariant  defined in \cite{MOO} depends only on the 
homological information which can be obtained from 
 the linking matrix of the surgery link.

Theorem 3 generalizes the well-known decompositon results for quantum invariants 
stated  in \cite{Le, Bl1}. We expect that these decomposition results extend
to refined TQFTs.

%  We also give a refined decomposition for \mbox{$\gcd(\delta,d/\delta)>1$}.

\vspace*{2mm}
\textbf{Organization of the paper.} After recalling the basic definitions of $\Spin^c$--structures, we give
homotopy theoretical and also combinatorial descriptions of their reduction modulo $d$. 
Further we recall  the definitions of ribbon and modular categories and  of quantum 3--manifold
invariants.
Section \ref{graded} deals with modular categories containing invertible objects. The  refined invariants are studied
in Sections \ref{refined} and 7. The last section is devoted to the proof of Theorem \ref{decomp}.\\

%\textbf{Aknowledgements.}

\section{Complex spin structures mod \emph{d}}

In this section, given a compact orientable 3--manifold $M$, 
we define the set $\Spin^c_d(M)$.

\vspace*{2mm}

%\noindent
%{\bf Conventions.} 
Throughout this paper
% $d$ is a fixed even integer;
 all manifolds
are assumed to be compact  and oriented; all (co)homology groups 
are computed with integer coefficients,
unless otherwise is specified; $\Z_x$ denotes
the cyclic group of integers
modulo $x$.
%$(a,b)$ denotes the greatest common divisor of integers $a$ and $b$.

 \subsection{$\Spin^c$--structures}
 \label{equiv}
 Complex spin structures are additional structures
 some manifolds can be endowed with and
 just like the more common spin structures, they
 can be seen as a generalization of orientations. 
 We recall basic facts and equivalent
 ways to define $\Spin^c$--structures following the 
 lines of \cite{Massuyeau2} and \cite{Ozbagci:2004}.
 
 Let $n\geq 1$ be an integer. The group $\Spin(n)$ is defined as
 the non-trivial double 
 cover of the special orthogonal
 group $\SO(n)$:
 $$1\longrightarrow \Z_2 \longrightarrow \Spin(n)\overset{\lambda}{\longrightarrow}
 \SO(n)\longrightarrow 1.$$  
\begin{exa} $\Spin(1)\cong\Z_2$, $\Spin(2)\cong S^1$ and
  $\Spin(3)\cong \SU(2).$
 \end{exa} 
 \noindent The complex spin group is defined
 as the quotient
 $$\Spin^c(n):=\frac{\Spin(n)\times S^1}{ \Z_2}$$ where
 $\Z_2$ is generated by  $(-1,-1)\in \Spin(n)\times S^1$.
 It follows that the map $\rho:\Spin^c(n)\to \SO(n)$, defined as
 $\rho([A,z])=\lambda(A)$ is
 a principal $S^1$--fibration. 
\begin{exa}
 $\Spin^c(3)\cong \U(2)$.
\end{exa}

Let $X$ be an $n$--dimensional Riemannian manifold and denote by $P_{\SO(n)}\to X$ the
principal bundle of oriented orthonormal frames of $X$. 
 \begin{dfn}[\cite{Ozbagci:2004}]
 \label{definition}
  A $\Spin^c$--structure on the manifold $X$ is a principal 
  $\Spin^c(n)$--bundle $P_{\Spin^c(n)}\overset{s}{\longrightarrow}X$, together 
  with a map $\pi:P_{\Spin^c(n)} \longrightarrow P_{\SO(n)}$ that restricted 
  to the fibers is $\rho$, i.e, makes the following diagram
  commute, i.e,
  \begin{equation*}
\xymatrix@R=6mm@C=3mm{
P_{\Spin^c(n)}\times \Spin^c(n)\ar[dd]^{(\pi, \rho)} \ar[rr]& & P_{\Spin^c(n)}  \ar[dd]^{\pi}\ar[dr]^{s} \\
& & & X\\
 P_{\SO(n)} \times \SO(n) \ar[rr] & &  P_{\SO(n)} \ar[ur]\\
}
\end{equation*}
 \end{dfn}

 Since $\rho:
 \Spin^c(n)\to \SO(n)$ is the 
 $S^1$--bundle of the unique non trivial 
 line 
 bundle over $\SO(n)$ (\cite{Milnor:1963})  and 
 isomorphism classes
 of principal $S^1$--bundles
 over $X$ are in
 one-to-one correspondence 
 with elements of
 $H^2(X;\Z)$, complex spin
 structures can be defined as cohomology classes:
 \begin{dfna}[\cite{Ozbagci:2004}]
 A $\Spin^c$--structure on the manifold $X$ is an element $\sigma\in H^2(P_{\SO(n)};\Z)$
 whose restriction to every fiber of $P_{\SO(n)}\longrightarrow X$ is the
 unique non trivial element of $H^2(\SO(n);\Z)\cong \Z_2$, i.e,
 $$\Spin^c(X)=\{\sigma\in H^2(P_{\SO(n)};\Z)| \sigma_{|{\text{fiber}}}\neq 0 \in H^2(\SO(n);\Z)\}.$$
\end{dfna}

Let $\BSO(n)$ denote  the Grassman 
manifold of oriented $n$-planes in
$\R^\infty$ and let 
$\gamma_{\SO(n)}$ be
the universal $n$-dimensional oriented vector 
bundle over $\BSO(n)$. Note that for 
discrete topological groups $G$, the classifying space $BG$ is Eilenberg-Maclane of type $(G,1)$.
Let $h$ be the unique (up to homotopy)
non-homotopically trivial map from $\BSO(n)$
to the Eilenberg--Maclane space $\K(\Z,3)$. We fix $h$ in its homotopy class
and we define the fibration
$\BSpin^c(n)\overset{\pi}{\rightarrow}\BSO(n)$ as the pull-back
under $h$
of the path-space fibration over $\K(\Z,3)$.  
If we set $\gamma_{\Spin^c(n)}:=\pi^*(\gamma_{\SO(n)})$, then:\\

%\AB{eventually mention that $BG=K(G,1)$}

\begin{dfna}
 A $\Spin^c$--structure on the
 manifold $X$ is a homotopy class of bundle maps
 between the (stable) tangent bundle $\TX$ of $X$
 and $\gamma_{\Spin^c(n)}$.
\end{dfna}

\subsubsection{The Chern map}
\label{Chern}
Let us denote by $\alpha:\Spin^c(n)\to S^1$
the homomorphism
 $\alpha([A,z])=z^2$. Then, to any $\Spin^c$--
structure $\sigma$ on an $n$--dimensional manifold $X$ (in the sense of Definition \ref{definition}),
one can associate a complex line bundle as follows:
using the map $\alpha$, the action of the group $\Spin^c(n)$ 
on the space $P_{\Spin^c(n)}$ extends to an action
on the product $P_{\Spin^c(n)}\times \mathbb{C}$ and we consider
its orbit space
$$\det(\sigma):
=(P_{\Spin^c(n)}\times \mathbb{C})/\Spin^c(n),$$ called the \emph{determinant line bundle} of $\sigma$.
The Chern map $$c:\Spin^c(X)\to H^2(X)$$ is defined as $c(\sigma):=c_1(\det(\sigma)),$
where $c_1$ is the first Chern class of the bundle $\det(\sigma)$, 
and it is affine over the doubling map
$H^2(X)\overset{\cdot 2}{\to} H^2(X).$
See \cite{Massuyeau2} for more details.

\subsection{$d$--complex spin structures}\label{dcomplex}
In this subsection we define
$d$--complex spin structures (short $\Spin^c_d$--structures)
on $n$--dimensional manifolds 
and we describe some
of their properties. Then,
we focus on dimension three and we present
a set of refined Kirby's moves for
$\Spin^c_d$--manifolds obtained by surgery 
along links in $S^3$.

Let $\beta:H^2(\BSO(n);\Z_2)\longrightarrow H^3(\BSO(n);\Z_d)$ be the 
Bockstein homomorphism associated to the exact sequence of groups:
$$0\longrightarrow\Z_d\overset{\cdot 2}{\longrightarrow}\Z_{2d}\longrightarrow \Z_2
\longrightarrow0.$$
The following lemma will help justify our construction of $\Spin^c_d$--structures.
 \begin{lem}
  The group $H^3(\BSO(n);\Z_d)$ is cyclic of order $2$
  generated by $\beta(w_2)$, where $w_2$ is the second Stiefel--Whitney class.
 \end{lem}

Since the elements of the group  $H^3(\BSO(n);\Z_d)\cong \Z_2$ are in
one-to-one correspondence with
homotopy classes of maps in $[\BSO(n); \K(\Z_d,3)]$,
there is a unique (up to homotopy) non-homotopically trivial map
$g:\BSO(n)\longrightarrow \K(\Z_d,3).$ We fix $g$ in its homotopy class 
and define the fibration $\pi_d:\BSpin^c_d(n)\longrightarrow \BSO(n)$
as the pull-back of the path space fibration of $\K(\Z_d, 3)$ under the map $g$:
\begin{equation*}
\xymatrix@R=8mm@C=5mm{
 \BSpin^c_d(n)\ar[d]^{\pi_d} \ar[r] & \PK(\Z_d,3) \ar[d] \\
 \BSO(n)\ar[r]^{g} & \K(\Z_d,3)  \\
}
\end{equation*}
We set $\gamma_{\Spin^c_d(n)}=\pi_d^*(\gamma _{\SO(n)})$.
Note that another choice of the map $g$ in
its homotopy class yields a different, but homotopy
equivalent, space $\BSpin^c_d(n)$.

\noindent Let $X$ be an $n$--dimensional Riemannian manifold.
\begin{dfn}
\label{bundlemap}
 A $d$--complex spin structure on $X$ is 
 a homotopy class of a bundle map between 
 the (stable) tangent bundle $\TX$ of $X$
 and $\gamma_{\Spin^c_d(n)}$.
\end{dfn} 
%\noindent Equivalently,

\begin{dfna}
\label{lift}
 A $d$--complex spin structure on $X$ is a
 homotopy class of a lift $\overline{f}$ of $f$ to $\BSpin^c_d(n)$, where
 $f:X\longrightarrow \BSO(n)$ is a classifying map for the
 bundle $\TX$.
 \end{dfna}
    
\begin{equation*}
\xymatrix@R=8mm@C=5mm{
              & \BSpin^c_d(n) \ar[d]^{\pi_d} \\
 X \ar[r]^{f} \ar@{.>}[ru]^{\overline{f}}& \BSO(n)} \\
\end{equation*}

Since the fiber of $\pi_d$ is the Eilenberg-Maclane space $\K(\Z_d,2)$,
there is a unique
obstruction $w_X$ to the existence 
of lifts $\overline{f}$ and this obstruction
lies in the group $H^3(X;\Z_d)$.

Note that the universal obstruction $w\in H^3(\BSO(n);\Z_d)$ (obtained from 
$w_X$ by setting $X=\BSO(n)$ and $f=\id_{\BSO(n)}$) is non-zero, therefore 
not all manifolds can admit $\Spin^c_d$--structures.
To see this, assume the contrary. Then, the fibration 
 $\pi_d:\BSpin^c_d(n)\longrightarrow \BSO(n)$ has a section
 $s$ and the map
 $k\circ s$ lifts $g$ to the contractible 
space $\PK(\Z_d,3)$. Therefore $g$ must be null-homotopic which contradicts
the choice we made for $g$.
%\begin{eqnarray*}
 %k\circ(h\circ s)&=&(k\circ h)\circ s\\
   %              &=&(g\circ\pi_d)\circ s\\
     %            &=&g,
%\end{eqnarray*}
 \begin{equation*}
\xymatrix@R=8mm@C=6mm{
 \BSpin^c(\Z_d)\ar[r]^{k}\ar[d]^{\pi_d} \ar[r] & \PK(\Z_d,3) \ar[d]^{p} \\
 \BSO(n)\ar@/^/[u]^{s}\ar[r]^{g} & \K(\Z_d,3)}  \\
\end{equation*}
As a consequence, the universal obstruction $w$ 
is the generator $\beta(w_2)$ of $H^3(\BSO(n);\Z_d)$.

\begin{pro}
 The set $\Spin^c_d(X)$ of $d$--complex spin structures 
 on a manifold $X$ is non-empty, if and only if, 
 $\beta(w_2(X))=0\in H^3(X;\Z_d)$.
 If non-empty, the set $\Spin^c_d(X)$ is affine over $H^2(X;\Z_d)$
 ($w_2(X)$ is the second Stiefel--Whitney class of $X$).
\end{pro}

\begin{proof}
 \begin{equation*}
\xymatrix@R=8mm@C=6mm{
              & \BSpin^c_d(n) \ar[d]^{\pi_d} \\
  X \ar[r]^{f} \ar@{.>}[ru]^{\overline{f}}& \BSO(n)} \\
\end{equation*}
Let $f$ be the classifying map of the (stable) tangent bundle of $X$. 
Then, the lifts $\overline{f}$ of
$f$ to $\BSpin^c(\Z_d)$ are in one-to-one correspondence with the sections 
of the bundle 
$f^*(\gamma_{\Spin^c_d(n)})$. 
The obstruction $w_X$ to the existence of such sections is a characteristic
cohomology class so by its naturality property we have
\begin{eqnarray*}
 w_X(f^*(\gamma_{\Spin^c_d(n)}))&=&f^*(w(\gamma_{\Spin^c_d(n)}))\\
                                  &=&f^*(\beta(w_2)).
                                  %&=&\beta(w_2(M))\in H^3(M;\Z_d).
\end{eqnarray*}
The above relation together with the naturality of Bockstein homomorphisms imply
the result. The second part of the theorem follows by standard arguments of 
obstruction theory.
\end{proof}\subsubsection{Restriction to the boundary}
Let us consider a manifold $X$ with boundary.
Then, any section of $TX|_{\partial X}$ 
transverse to
$\partial X$ and oriented outwards 
gives rise to a homotopy class 
of isomorphisms between the 
oriented vector bundles
 $\R\oplus \emph{T} \partial X$ and $TX|_{\partial X}$.
 Therefore, there is a well-defined restriction map
 $$\Spin^c_d(X)\to \Spin^c_d(\partial X)$$
 affine over the map 
 $H^2(X;\Z_d)\overset{i^*}{\to} H^2(\partial X;\Z_d)$
 induced by the inclusion $i:\partial X\to X$.
\subsubsection{From $\Spin^c$ to $\Spin^c_d$}
We have seen in Section \ref{equiv}
that the fibration $\BSpin^c(n)\overset{\pi}{\rightarrow}\BSO(n)$
is defined as the pull-back under $h$ of the path-space
fibration over $\K(\Z,3)$.
The modulo $d$ restriction morphism $\xi: \Z\to \Z_d$ induces
a map $\K(\Z,3)\overset{\xi_*}{\to} \K(\Z_d,3)$ on the level of Eilenberg-Maclane
spaces and further,
via $\xi_*$, a natural map
$\BSpin^c(n)\to \BSpin^c_d(n)$. As a result, there exists a well-defined natural map
$$\Spin^c(X)\to \Spin^c_d(X)$$ affine
over $H^2(X)\overset{\xi^*}{\to} H^2(X;\Z_d)$ induced by $\xi$.
 \begin{equation*}
\xymatrix@R=6mm@C=3mm{
& & & \BSpin^c(n)\ar[dr]\ar[rr]\ar[dd]^{\pi}&  & \PK(\Z,3)\ar @{->}[dd]|!{[d];[d]}\hole \ar[dr]& \\
 & & & & \BSpin^c_d(n)\ar[dd]^(.65){\pi_d}\ar[rr] &&\PK(\Z_d,3)\ar[dd]\\
 X\ar[rrr]^{f}& & & \BSO(n) \ar @{=}[dr] \ar @{->}[rr]^(.35){h}|!{[r];[r]}\hole& &\K(\Z,3)\ar @{->}[dr]^{\xi_*}&  \\
& & & & \BSO(n) \ar[rr]^{g} & & \K(\Z_d,3)}
\end{equation*}

\subsection{Combinatorial description of $\Spin^c_d$--structures}
\label{combinatorial}
Let $L=(L_1,\cdots, L_n)$ be an oriented framed link in $S^3$
with linking matrix $( L_{ij} )_{i,j=1,n}$ and
denote by $S^3(L)$ the $3$--manifold obtained by surgery. 
The manifold $S^3(L)$ is the boundary of
a $4$--manifold $W_L$,
constructed from a $4$--ball $D^4$ by
attaching $n$ $2$--handles $(D^2\times D^2)_i$ along embeddings 
of $-(S^1\times D^2)_i$ in concordance with the orientation
and framing of each component $L_i$. The $4$--manifold 
$W_L$ is sometimes called \emph{trace of surgery}.

\noindent Let us define the set
$$\mathcal{S}^c_d(L)=\frac{\{(\sigma_1,\dots,\sigma_n)\in 
(\Z_{2d})^n| \sigma_i=L_{ii}\pmod 2 \}}{ 2 \Im L}$$
whose elements will be called \emph{modulo} $d$ \emph{Chern vectors}.

\begin{thm}\label{ScdK}
 There is a canonical bijection
 $$\Spin^c_d(S^3(L))\overset{\phi_L}{\longrightarrow} \mathcal{S}^c_d(L).$$
\end{thm}
\begin{proof}
We have seen in the previous section that $\Spin^c$--structures induce $d$--complex spin structures. In particular,
the map $r_d:\Spin^c(W_L)\longrightarrow \Spin^c_d(W_L)$ is surjective since
it is affine over the surjective map $H^2(W_L;\Z)\longrightarrow H^2(W_L;\Z_d)$ 
induced by restriction modulo $d$ of coefficients and
similarly, the restriction map 
$r:\Spin^c_d(W_L)\longrightarrow \Spin^c_d(S^3(L))$ is surjective 
since it is affine over the surjective map 
$H^2(W_L;\Z_d)\longrightarrow H^2(S^3(L);\Z_d)$ induced by inclusion. 

With the help of $r$ and $r_d$, we define the map $\phi_L:\Spin^c_d(S^3(L))\longrightarrow \mathcal{S}^c_d(L)$ as follows:
any $\sigma\in \Spin^c_d(S^3(L))$ is the image of an
element $\tilde{\sigma}\in \Spin^c(W_L)$ under the composition
$r\circ r_d$. The value $c(\tilde{\sigma})\in H^2(W_L;\Z)\cong \Z^n$ 
is characteristic for $L$ (see \cite{Massuyeau2}), therefore $c(\tilde{\sigma})=(c_1,\dots,c_n)\in \Z^n$
with $c_i=L_{ii} \pmod 2$. We set $\phi_L(\sigma)$ to be the image of
$c(\tilde{\sigma})$ $\pmod {2d} $ in $\mathcal{S}^c_d(L)$. 
\begin{itemize}
 \item $\phi_L(\sigma)$ is well-defined:
 let us assume that $(r\circ r_d)^{-1}(\sigma)$ contains
 two different elements $\tilde{\sigma_1}$ and $\tilde{\sigma_2}$.
 Then, they differ by $y\in H^2(W_L)\cong \Z^n$ whose modulo $d$
 reduction belongs to $\Im L$ in $(\Z_d)^n$.
 Since $2y$ $\pmod {2d}$ belongs to $2\Im L$ in $(\Z_{2d})^n$, we get that
 $[c(\tilde{\sigma_2})]=[c(\tilde{\sigma_1})+2y]=[c(\tilde{\sigma_1})]$
 (the first equality follows from the fact that the Chern map is affine over the doubling map).

 \item $\phi_L$ is injective: let us assume that $\phi_L(\sigma_1)=\phi_L(\sigma_2)$.
 Then, the preimages $\tilde{\sigma_1}$ and $\tilde{\sigma_2}$  (of $\sigma_1$ and
 $\sigma_2$, respectively) under
 $r\circ r_d$ differ by $y\in H^2(W_L)\cong \Z^n$ such
 that $2y$ $\pmod{2d}$ belongs to $2\Im L$ in $(\Z_{2d})^n$. This implies
 that $y$ $\pmod d$ belongs to $\Im L$ in $(\Z_d)^n$ and therefore 
 $\sigma_1=\sigma_2$.
\end{itemize}
The sets $\Spin^c_d(S^3(L))$ and $\mathcal{S}^c_d(L)$ have
the same cardinality, hence $\phi_L$ is bijective.
\end{proof}

\noindent From now on, we will reffer to the set $\mathcal{S}^c_d(L)$ as the combinatorial 
description of $d$--complex spin structures on the surgered manifold $S^3(L)$.

\subsection{$\Spin^c_d$ Kirby moves} 
A celebrated theorem of Kirby \cite{Kirby} states that two
(oriented) framed links in $S^3$ produce the same manifold
by surgery, if and only if, they are related 
by a finite sequence of local geometric transformations called 
\emph{Kirby moves}.
In what follows, we present a refined version of the original 
Kirby theorem for manifolds equipped with $\Spin^c_d$--structures.

\begin{figure}[htb]
\begin{center}
\includegraphics{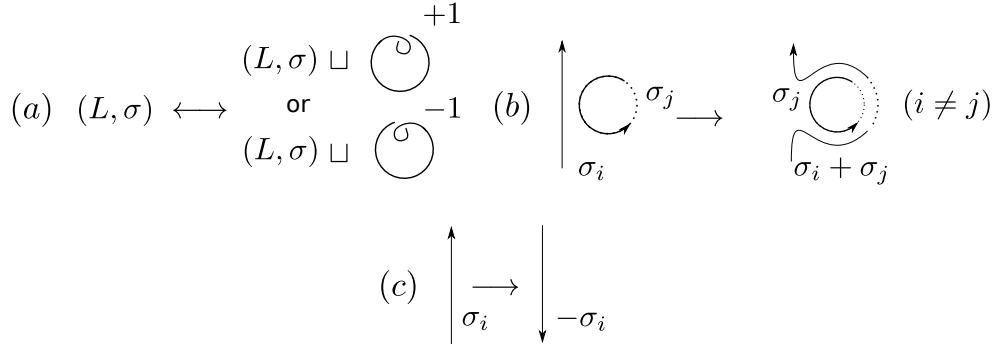}
\caption{Refined $\Spin^c_d$--Kirby moves (a) Stabilization (b) Handle slide (c) Orientation reversal. Note that 
we use the blackboard framing and the labels refer to modulo $d$ Chern vectors.}
\label{k1}
\end{center}
\end{figure}
\begin{thm}\label{Kirby}
 Let $(L,\sigma)$ and $(L',\sigma')$ be two oriented framed links
 with Chern vectors $\sigma$, $\sigma'$. %corresponding
 %to $\Spin^c_d$-structures $s$ and $s'$.
 Then, the manifolds $(S^3(L),\sigma)$ and $(S^3(L'),\sigma')$ 
 are $\Spin^c_d$--homeomorphic, if and only if,
 the pairs $(L,\sigma)$ and $(L',\sigma')$ are related
 by a finite sequence of the moves in Figure \ref{k1} or their inverses.

\end{thm}

%\AB{Letter s should be reserved for Spin structures, can not we identify
%$\Spin^c$-structures with their chern vectors here?? The same applies for Thm.3.2 below}

\begin{proof}
 We must
 check that for any of the Kirby moves (stabilization, handle slide and
 orientation reversal) $L\overset{K}{\rightarrow} L'$,
 the Chern vectors change under the homeomorphism 
 $S^3(L)\rightarrow S^3(L')$ in the way described by Figure \ref{k1}.
 To do this, note that the map $r\circ r_d:\Spin^c(W_L)\to 
 \Spin^c_d(S^3(L))$ defined in the proof of Theorem \ref{ScdK} is surjective and the 
 following diagram commutes:
\begin{equation*}
\xymatrix@R=8mm@C=5mm{
 \Spin^c(W_L)\ar[d]^{r\circ r_d} \ar[rr]^{K}& & \Spin^c(W_{L'}) \ar[d]^ {r'\circ r'_d}\\
 \Spin^c_d(S^3(L))\ar[rr]^{K} & & \Spin^c_d(S^3(L')). \\
}
\end{equation*}
\noindent Since $\Spin^c$--structures on $W_L$ are combinatorially described as 
 the elements of the set $$\{(c_1,\cdots, c_n)\in \Z^n\, |\,  c_i=L_{ii} \pmod 2\}$$ and their change
 under Kirby moves is known (\cite{Massuyeau2}),
 the conclusion follows.
\end{proof}

\section{Generalized spin structures}

In \cite{Bl1} the second author introduced spin structures 
modulo an even integer 
$d$ (short $\Spin_d$--structures). In dimension three,
he gave a combinatorial description of such structures
as well as a refined set of Kirby moves. In this section
we recall his results using the notations of Section \ref{combinatorial}
and extend the definition to a possibly non cyclic group of coefficient $K$
with distinguished order $2$ element $v\in K$.

\subsection{Combinatorial description of $\Spin_d$--structures}
Consider the set
 $$\mathcal{S}_d(L)=\{(s_1,\dots, s_n)\in (\Z_d)^n \;|\;
 \sum^n_{j=1}L_{ij} s_j=\frac{d}{2} L_{ii} \pmod d \}.$$
 The elements of $\mathcal{S}_d(L)$ are called \emph{modulo} $d$
\emph{characteristic solutions of} $L$. 

 \begin{lem}
  There is a canonical bijection $$\Spin_d(S^3(L))\overset{\psi_L}{\longrightarrow} \mathcal{S}_d(L).$$
 \end{lem}
\begin{proof}
 Given a $\Spin_d$--structure $\sigma$ on
 $S^3(L)$ it can be extended to $W_L$ if and
 only if a certain cohomology class 
 in $H^2(W_L,S^3(L);\Z_d)$ vanishes.
 We denote this class by $w(W_L;\sigma)$ and we 
 call it relative
 obstruction. 
 To any $\sigma\in \Spin_d(S^3(L))$ there
 is associated a relative obstruction 
 $w(W_L;\sigma)$ in $H^2(W_L, S^3(L);\Z_d)$.
 Since the group $H^2(W_L, S^3(L);\Z_d)$ is free of
 rank $n$, taking the coefficients of
 the relative obstruction we obtain a map
 $\psi_L:\Spin_d(S^3(L))\longrightarrow (\Z_d)^n$.
 We will show that this map is injective and its image 
 coincides with $\mathcal{S}_d(L)$. Let us consider 
 the embedding $\eta:\Z_2\hookrightarrow \Z_d$. Then,
 the relative obstruction $w(W_L;\sigma)=\eta_*(w_2(W_L))$,
 where $\eta_*$ is the induced map on the level of cohomology.
 Given an integral $2$--cycle $x$, we denote by $x\cdot x$ its
 self-intersection number and by $[x]_m$ its modulo $m$ restriction.
 The second Stiefel--Whitney class $w_2(W_L)$ is defined by the 
 following equation:
 $$\la w_2(W_L), [x]_2\ra=x\cdot x \pmod 2, \forall x$$
 therefore, the relative obstruction is defined
 by $$\la w(W_L;\sigma), [x]_d\ra=\frac{d}{2} x\cdot x \pmod d, \forall x.$$
 Using functoriality and writing $w(W_L;\sigma)$ in the preferred basis of $H^2(W_L;\Z_d)$,
the result follows.
\end{proof}
\noindent The set $\mathcal{S}_d(L)$ 
will be referred to as a combinatorial 
description of $\Spin_d$--structures on
the surgered manifold $S^3(L)$.

 \subsection{ $(K,v)$-spin structures}\label{kv}
 Let $K$ be a finite abelian group and $v$ a non trivial element in the $2$-torsion subgroup $\mathrm{Tor}_2(K)$, then we can define $(K,v)$ generalized spin structures.
\begin{dfn}
For $n\geq 3$, an  $(K,v)$ generalized spin structure on an $SO(n)$ principal bundle $SO(n)\hookrightarrow P\rightarrow M$ is a cohomology
class $\sigma\in H^1(P,K)$ whose restriction to the fiber is equal to $v$ in $H^1(SO(n),K)\cong \mathrm{Tor}_2(K)$.
 A $(K,v)$ generalized spin structure on an oriented manifold 
 of  dimension $\geq 3$ is a $(K,v)$
 spin structure on its oriented framed bundle.
\end{dfn} 

As usual this definition can be extended to dimensions less than 3 by using stabilization.

If $K$ is decomposed as a product of cyclic groups, $
K=\Z_{d_1}\times \dots \times \Z_{d_k}$, then a  $(K,v)$ generalized spin structure is a sequence $(\sigma_1,\dots, \sigma_k)$ where $\sigma_j$ is either a $d_j$--spin 
structure or a mod $d_j$ cohomology class,
depending on the twist coefficient of the corresponding generator.

\subsection{$\Spin_d$ Kirby moves}\label{spinKirby}
 \begin{figure}[h!]
\begin{center}
\includegraphics{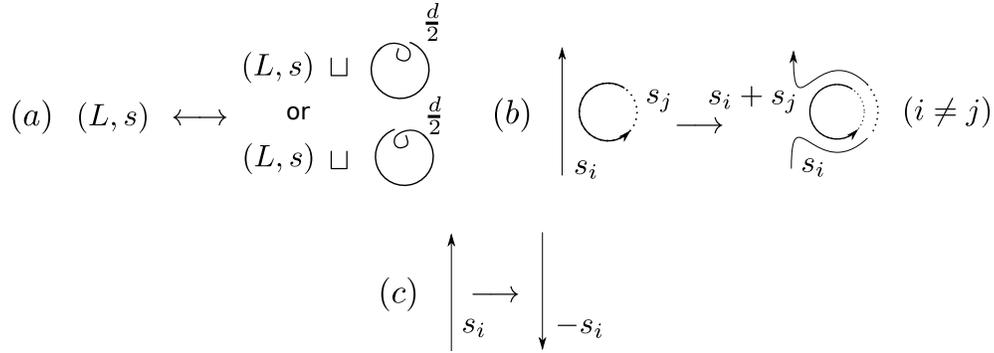}
\end{center}
\caption{\label{k2}Refined $\Spin_d$--Kirby moves (a) Stabilization (b) Handle slide (c) Orientation reversal.
Note that we use the blackboard framing and the labels
refer to modulo $d$ characteristic solutions of $L$.}
\end{figure}
The second author proved the following result:
\begin{thm} \cite{Bl1}
 Let $(L,s)$ and $(L',s')$ be two oriented framed
 links with characteristic solutions $s,s'$. Then, the
 manifolds $(S^3(L), s)$ and $(S^3(L'), s')$ are
 $\Spin_d$--homeomorphic, if and only if, the pairs
 $(L,s)$ and $(L',s')$ are related by a finite sequence
 of the moves in Figure \ref{k2} or their inverses.
 \end{thm}

\section{From ribbon to modular categories}
\subsection{Ribbon categories}\label{ribcat}
We introduce basic notions about ribbon categories following 
\cite{Tbook}. Without loss of generality, we work
with monoidal categories that are strict (according to Mac Lane's coherence theorem \cite{Maclane}, every
monoidal category is equivalent to a strict one), i.e, categories with tensor
product and unit object $\1$ satisfying
\begin{eqnarray*}
 V\otimes \1=\1\otimes V=V\\
 (U\otimes V)\otimes W=U\otimes (V\otimes W)\\
 f\otimes \id_\1=\id_\1\otimes f=f\\
 (f\otimes g)\otimes h=f\otimes (g\otimes h)
\end{eqnarray*}
for any objects $U,V,W$ and morphisms $f,g,h$ of the category.
%\subsubsection{Braiding and twist}

 A \emph{braiding} in a monoidal category $\C$
 is a family of natural isomorphisms 
 $b=\{b_{V,W}:V\otimes W\to W\otimes V\}$
 where $V,W$ run over all the objects of $\C$,
 such that $b_{U,V\otimes W}=(\id_V\otimes b_{U,W})(b_{U,V}\otimes \id_W)$
 and $b_{U\otimes V,W}=(b_{U,W}\otimes \id_V)(\id_U\otimes b_{V,W})$.
%\AB{There is no category of n-dim. vector spaces, there is a category of finite dimensional vector spaces, whose objects are vector
%spaces of different dimensions.Please correct it everywhere}

%\begin{exa}
%\label{braiding}
Let us consider the category $(\Vect_\k,\otimes)$ of finite dimensional
vector spaces over 
a field $\k$. The family of
maps $b_{V,W}:V\otimes W\to W\otimes V, \ b_{V,W}(v\otimes w)=w\otimes v$ for all $v\in V, w\in W$
defines a braiding.
%\end{exa}
%\begin{dfn}
 
A \emph{twist} in a monoidal category $\C$ with braiding $b$ is a family of
 natural isomorphisms $\theta=\{\theta_V:V\to V\}$ where $V$ runs over all the objects 
 of $\C$, such that for any two objects $V,W$ of $\C$, we have
 $\theta_{V\otimes W}=b_{W,V}b_{V,W}(\theta_V\otimes \theta_W).$
%\end{dfn}
%\begin{exa}
%\label{twist}

In the category $(\Vect_\k,\otimes,b)$,
the family $\{\theta_V\}_V$ of maps $\theta_V:V\to V, \theta_V=\id_V$
defines a twist.
%\end{exa}
%\subsubsection{Duality}

Let $\C$ be a monoidal category. Assume that to each 
 object $V$ of $\C$ there are associated an object 
 $V^*$ and two morphisms $c_V:\1\to V\otimes V^*,\  d_V:V^*\otimes V\to \1$.
% \begin{dfn}
 The rule $V\mapsto (V^*,c_V,d_V)$ is called \emph{duality} in $\C$ if $(\id_V\otimes d_V)(c_V\otimes \id_V)=\id_V$
 and $(d_V\otimes \id_{V^*})(\id_{V^*}\otimes c_V)=\id_{V^*}$ for all objects $V$ of $\C$.
 The duality is called \emph{compatible} with the braiding $b$ and the twist $\theta$
 if in addition, for any object $V$ of $\C$ we have
 $(\theta_V\otimes \id_{V^*})c_V=(\id_V\otimes \theta_{V^*})c_V.$
%\end{dfn}
%\begin{exa}

For example, in $(\Vect_\k,\otimes)$,  if we set $V^*:=\Hom_\k(V,\k)$, fix a basis $\{e_i(V)\}_{i=1,\dots,n}$ of $V$ 
with $n= \dim(V)$, 
and  denote by $\{e_i(V)^*\}_{i=1,\dots,n}$ its dual,
we can define the map $c_V:\k\to V\otimes V^*$ by  $c_V(1)=\sum_{i=1}^{n}e_i(V)\otimes e_i(V)^*$ and
$d_V:V^*\otimes V\to \k$ by $d_V(f,v)=f(v)$ then, the rule $V\mapsto (V^*,c_V,d_V)$
defines a duality in $\Vect_\k$, compatible
with the braiding $b$ and twist $\theta$ defined above. 
%Examples \ref{braiding}, \ref{twist}.
%\end{exa}
%\subsubsection{Ribbon categories, traces and dimensions}
\begin{dfn}
A monoidal category $\C$ equipped 
 with braiding $b$, twist $\theta$ and a compatible duality
 $(*,c,d)$ is called \emph{ribbon}.
 \end{dfn}
 \begin{exa}
 \begin{enumerate}
 \item The category $(\Vect_\k,\otimes, b, \theta, (*,c,d))$ defined above
 is ribbon;
 \item Let ${\bf U}_q(\mathfrak g)$ be the quantum group corresponding to a 
semi-simple Lie algebra $\mathfrak g$, then the category of its finite 
dimensional representations
$\Rep({\bf U}_q(\mathfrak g))$ is ribbon \cite{Tbook}. 
\end{enumerate}
\end{exa}
\noindent In ribbon categories we can
define traces of morphisms and dimensions of objects as follows.
\begin{dfn}\label{qtrace}
 \begin{enumerate}[(1)]
  \item The \emph{quantum trace} of an endomorphism $f:V\to V$
  is the composition $\tr(f):=d_Vb_{V,V^*}((\theta_Vf)\otimes \id_{V^*})c_V:\1\to \1;$
  \item The \emph{quantum dimension} of an object $V$ is 
  $\la V \ra:=\tr(\id_V).$
 \end{enumerate}
\end{dfn}

\begin{pro}{\rm In any ribbon category $\C$
\begin{enumerate}[(a)]
\item $\tr(fg)=\tr(gf)$
for any morphisms $f\in \Hom(V,W)$ and $g\in \Hom(W,V)$;

\item $\tr(f\otimes g)=\tr(f)\tr(g)$ for any
morphisms $f\in \Hom(V,W)$ and $g\in \Hom(V',W')$;

\item $\langle V\otimes W \rangle=\langle V\rangle\langle W\rangle$ 
for any two objects $V,W$;
\item $\la V\ra =\la V^*\ra$ for any object $V$.
\end{enumerate} }
\end{pro}

%\begin{exa}
%\begin{enumerate}[(1)]
%\item If $f\in \End(1)$, then $\tr(f)=f$;
%\item $\langle1\rangle=1_\k$.
%\end{enumerate}
%\end{exa}

\subsubsection{The category of colored ribbon graphs}

A ribbon graph is a compact, oriented surface in $\R^2\times [0,1]$
decomposed into bands, annuli and 
coupons. Bands are homeomorphic images of the square $[0,1]\times [0,1]$,
annuli are homeomorphic images of the cylinder
$S^1\times [0,1]$ and coupons are bands
with a distinguished base, called bottom (the 
oposite base is called top).
Isotopy of ribbon graphs means isotopy in the strip
$\R^2\times [0,1]$, constant on the boundary intervals
and preserving the decomposition
into annuli, bands and coupons as well
as preserving orientations.

Given a ribbon category $\C$, a ribbon graph is 
$\C$--\emph{colored} if each band and each 
annulus is equipped with
an object of $\C$, called \emph{color},
and each coupon 
is colored by a morphism
in $\C$.
Isotopy of colored ribbon graphs
means color-preserving isotopy.
Starting with a ribbon category $\C$, one can 
construct the category $\Rib_\C$ of $\C$--colored ribbon 
graphs as follows: the objects are
finite sequences of end points of 
ribbon graphs colored with objects of $\C$ or
their duals, according to the orientation of strands, and
the morphisms of $\Rib_\C$ are isotopy classes of 
$\C$--colored ribbon graphs (see Figure \ref{mf}).
The category $\Rib_\C$ can 
be given a tensor product, a natural braiding, twist
and compatible duality and it becomes in this way ribbon.
\begin{figure}[htbp]
\begin{center}
\includegraphics{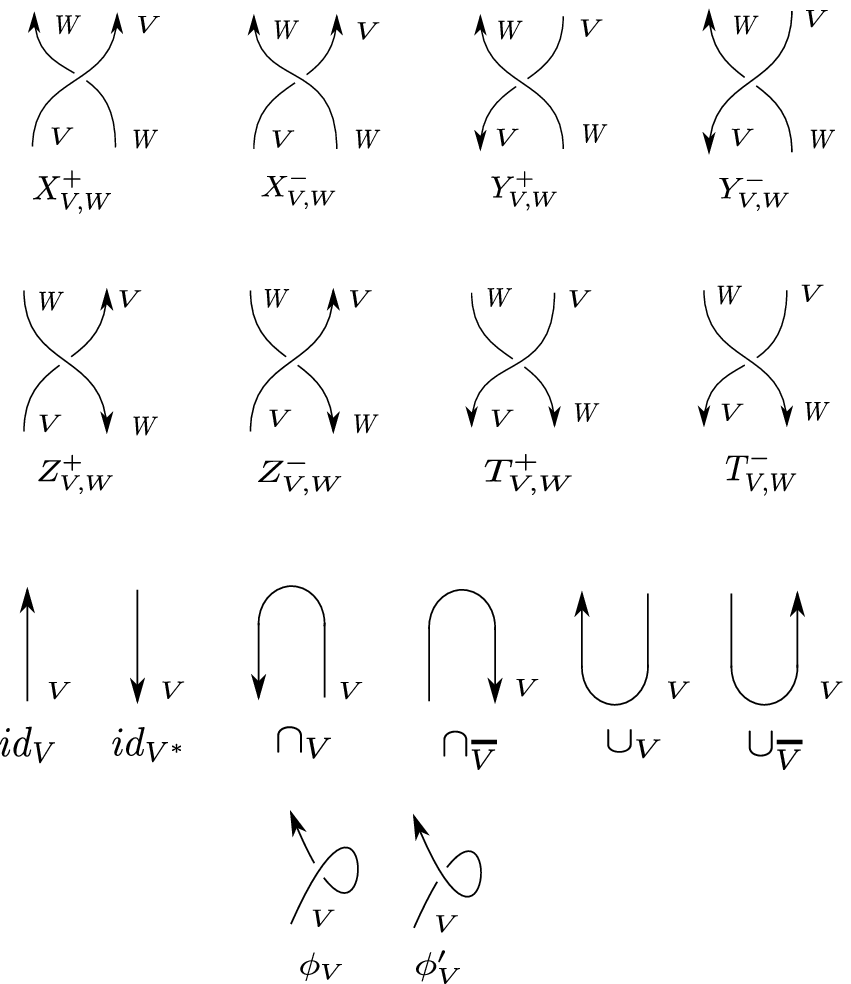}
\end{center}
\caption{Morphisms in $\Rib_\C$ \label{mf}}
\end{figure}

Ribbon categories turn out to play an important 
role in the theory of link invariants. The following 
theorem, due to Turaev, describes 
a way to associate to any ribbon category, 
invariants of colored ribbon graphs and
in particular, of colored ribbon tangles (i.e. graphs
without coupons).

\begin{thm}[\cite{Tbook}]
\label{functor}
 Let $\C$ be a ribbon category with a braiding
 $b$, a twist $\theta$ and a compatible duality 
 $(*,c,d)$. Then there
 exists a unique covariant functor $F_\C:\Rib_\C\to \C$ 
 preserving the tensor product such that (see Figure \ref{mf}):
 \begin{enumerate}
  \item  $F_\C(V,+1)=V$, $F_\C(V,-1)=V^*$;
  \item for any objects $V,W$ of $\C$, we have
  $F_\C(X^+_{V,W})=b_{V,W};F_\C(\phi_V)=\theta_V;F_\C(\cup_V)=c_V; F_\C(\cap_V)=d_V;$
  \item for any elementary $\C$--colored ribbon graph (ribbon graph with one coupon) $\Gamma$, we have 
  $F_\C(\Gamma)=f$ where $f$ is the color of the only coupon of $\Gamma$.
 \end{enumerate}
The functor $F_\C$ has the following properties:
$$ F_\C(X^-_{V,W})=(b_{W,V})^{-1}, F_\C(Y^+_{V,W})=(b_{W,V^*})^{-1}, F_\C(Y^-_{V,W})=b_{V^*,W}$$
$$ F_\C(Z^+_{V,W})=(b_{W^*,V})^{-1}, F_\C(Z^-_{V,W})=b_{V,W^*},$$
$$ F_\C(T^+_{V,W})=b_{V^*,W^*}, F_\C(T^-_{V,W})=(b_{W^*,V^*})^{-1}, F_\C(\phi'_V)=\theta_V^{-1}.$$
\end{thm}
\begin{figure}
\begin{center}
 
\end{center}
\end{figure}

In what follows we will use pictures to illustrate identities
in a ribbon category.
When drawing a pictorial identity, we always mean 
the corresponding morphisms in the ribbon category (see Figure \ref{formula} for an example).
\begin{figure}[htbp]
\begin{center}
\psfrag{a}{$V$}
\psfrag{b}{$=$}
\psfrag{o}{$f$}
\psfrag{e}{$g$}
\includegraphics{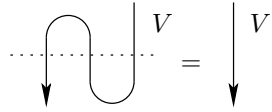}
\end{center}
\caption{$(d_V\otimes \id_{V^*})(\id_{V^*}\otimes c_V)=\id_{V^*}$\label{formula}}
\end{figure}
\begin{cor}[\cite{Tbook}]
\label{or}
 Let $\Omega$ be a $\C$-colored ribbon graph that contains an annulus
 $a$. If $\Omega'$ is the ribbon graph obtained from
 $\Omega$ by reversing the orientation of $a$ and replacing
 the color of $a$ with its dual object,
 then $F_\C(\Omega')=F_\C(\Omega).$
\end{cor}

\subsection{Modular categories}\label{modcat}
Let $\k$ be a field of zero characteristic. A ribbon category
is called $\k$--linear or pre-additive if the Hom 
sets are $\k$--vector 
spaces, composition and tensor product are bilinear
and $\End(\1)=\k$. An object $\l$ of the category is called
\emph{simple} if the map $u\mapsto u\cdot \id_{\l}$ from $\k=\End(\1)$ to 
$\End(\l)$ is an isomorphism.

We will denote by $\C^\oplus$ the additive closure of a
pre-additive ribbon category $\C$, which admits direct sums of
objects and  compositions of morphisms are modelled on the
matrix multiplication (\cite{Maclane}).
%\AB{Here we should cite e.g. Mac Lane book ``Categories for the working mathematician'', GTM Springer, 2. edition}

\begin{dfn} A modular category over the field $\k$ is
a $\k$--linear ribbon category $\C$ with a 
finite set of simple objects $\G$
that satifies the following axioms:
\begin{enumerate}[(1)]
\item \emph{Normalization}: the trivial object $\1$ is in $\G$;
\item \emph{Duality}: for any object $\l\in \G$, its dual $\l^*$ is isomorphic
to an object in $\G$;
\item \emph{Domination}: for any object $V$ of the category, there exists a finite 
decomposition $$\id_V=\sum_i f_i\id_{\l_i}g_i$$
with $\l_i\in \G, \forall i$;
\item \emph{Non-degeneracy}: the matrix
$S=(S_{\l\m})_{\l,\m\in \G}$ is
invertible over $\k$,
where $S_{\l\m}\in \k$ is the endomorphism of the trivial 
object associated with the $(\l,\m)$--colored, $0$--framed 
Hopf link with linking $+1$.
\end{enumerate}
\begin{figure}[htbp]
\begin{center}
 \psfrag{a}{$\l$}
 \psfrag{b}{$\mu$}
 \includegraphics{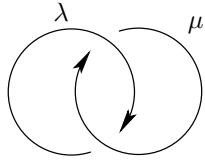}
 \caption{The Hopf link with linking number $+1$ and colors $\lambda$ and $\mu$}
\end{center}
\end{figure}
\end{dfn}
\noindent The category is called pre-modular if we remove the 
last axiom. Replacing $\C$ with its additive closure, we can reformulate the
domination axiom as follows: for any $\l,\nu, \mu \in\G$ there exist
positive integers $C^\mu_{\l\nu}$ called \emph{structure constants},
such that $\l\otimes\nu\simeq\oplus_{\mu\in \G}\;  C^\mu_{\l\nu}\;\mu$.
The domination axiom says that any object decomposes into a direct sum
of simple ones.
%It follows that $\G$ is a representative set for the 
%isomorphism classes of simple objects. 
\begin{exa}
Let $\mathfrak g={\mathfrak{sl}}_N$ and $q$ be an $(N+K)$th root of unity.
Then, there exists an associated modular category
$\C(N,K)$ ($N,K\geq 2$) with simple objects given by partitions $\l$ from
$$\G=\{\l=(\l_1,...,\l_s)\;| \;\l_1\leq K, s<N   \}\, . $$
See e.g. \cite{Blanchet} for more details.
\label{mainex1}
\end{exa}
 \subsection{Properties of pre-modular categories}\label{premod}
In what follows, we describe basic properties 
of pre-modular categories that will be relevant to the rest of
the paper. We follow the lines of \cite{BB}.
Unless otherwise stated, $\C$ is a 
pre-modular category over a field $\k$ with zero characteristic and $\Gamma$ is the set 
of representatives of simple objects.

\begin{dfn}
\begin{enumerate}[(a)]
 \item An object $\l$ of $\C$ is called \emph{transparent} if 
for any object $\nu$ the following
morphisms in $\C$ are equal
\begin{center}
\psfrag{a}{$\lambda$}
\psfrag{b}{$\nu$}
\psfrag{c}{$=$}
\includegraphics{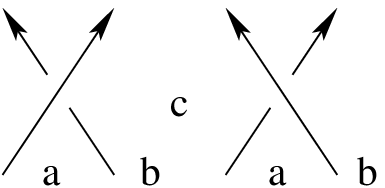}
\end{center}
\item A morphism $f\in \Hom(\l,\nu)$ is called \emph{negligible}
if $\tr(fg)=0$ for any $g\in \Hom(\nu,\l)$.
\item
 The \emph{braiding coefficients} between two objects $\l,\n \in \G$ are defined as 
 a collection $\{b_{\l, \nu}^{\mu}\}$ for all $\mu\subset \l\otimes \nu$ such that
 \begin{center}
 \includegraphics{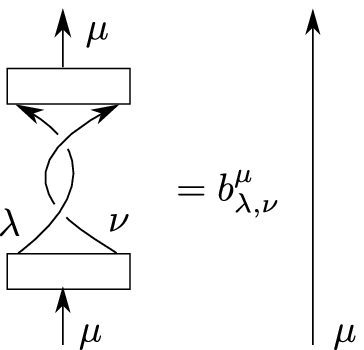}
 \end{center}
 \item For any object $\l\in \G$ the \emph{twist coefficient} $\theta_{\lambda}$ is defined
 by the equality:
 \begin{figure}[htbp]
 \label{twist}
  \begin{center}
  \psfrag{a}{$\lambda$}
  \psfrag{b}{$=\theta_{\lambda}$}
   \includegraphics{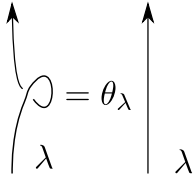}
  \end{center}
  \caption{Twist coefficient}
\end{figure}

\end{enumerate}
\end{dfn}
%\subsubsection{The fusion formula}
We discuss in detail 
a formula that will be used extensively throughout the paper.
Let us start by fixing some notations.
Given objects $\alpha_1,\cdots, \alpha_n, \beta_1,\cdots,\beta_m\in \Gamma$,
we denote by $H_{\alpha_1,\cdots,\alpha_n}^{\beta_1,\cdots,\beta_m}$ the
$\k$--module of morphisms $\Hom_\C(\alpha_1\otimes\cdots \otimes \alpha_n, \beta_1\otimes \cdots \otimes\beta_m).$

The modules $H_{\mu}^{\l\nu}$, $H_{\mu\nu^*}^{\l}$, $H_{\nu^*}^{\mu^*\l}$, $H_{\nu^*\l^*}^{\mu^*}$,
$H_{\l^*}^{\nu\mu^*}$ and $H_{\l^*\mu}^{\nu}$ 
are mutually isomorphic, as well as the modules $H_{\mu\nu^*\l^*}$, $H^{\l\nu\mu^*}$
and all obtained from them by a cyclic permutation of colors. 
Identifying these modules along the isomorphisms, we get a 
symmetrized multiplicity module $\tilde{H}^{\l\nu\mu^*}$ for
which only the cyclic order of colors is important. 
The elements of $\tilde{H}^{\l\nu\mu^*}$ are represented by
a round coupon with one incoming line (colored with $\mu$) and
two outgoing ones (colored by $\l$ and $\nu$).

The pairing $\langle\ ,\ \rangle:\tilde{H}^{\l\nu\mu^*}\times \tilde{H}^{\mu\l^*\nu^*}\to \k$
defined as $\langle f,g\rangle=\tr(fg)$
is non-degenerate since the category 
$\C$ can be assumed without negligible morphisms (if any, they can be quotiened
out). The symmetrized modules $\tilde{H}^{\l\nu\mu^*}$ and
$\tilde{H}^{\mu\l^*\nu^*}$ are dual to each other,
therefore we can 
choose bases $(a_i)_{i\in I_{\l\nu}^{\mu}}$
for $\tilde{H}^{\mu\l^*\nu^*}$ and $(b_i)_{i\in I_{\l\nu}^{\mu}}$ for 
the module $\tilde{H}^{\l\nu\mu^*}$ that are dual with respect to $\langle \ ,\ \rangle.$
Note that the composition $b_j\circ a_i$ is an endomorphism of the simple object $\mu$, so it 
is of the form $c_{ij}\cdot \id_{\mu}$, $c_{ij}\in \k$. 
Comparing the traces gives
$\delta_{ij}=c_{ij}\langle \mu \rangle $ so $b_j\circ a_i=\delta_{ij}\langle \mu \rangle^{-1}\id_{\mu}$.
For any simple objects $\l, \nu\in \G$, the domination axiom applied to
$\id_{\l\otimes \nu}$ yields the following relation, known as the \emph{fusion formula}:
$$\psdiag{14}{30}{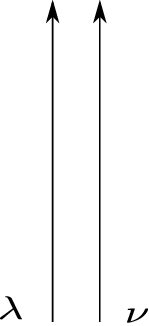} =\sum_{\mu\in \Gamma}\sum_{i\in I^{\mu}_{\lambda \nu}}\langle \mu \rangle \psdiag{14}{30}{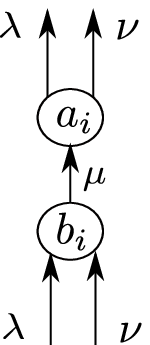}.$$

For any pre-modular category $\C$, let us define
the \emph{Kirby color}
$\omega= \sum_{\l\in\G}\la \l\ra \l$
as an element of $K_0(\C^\oplus)\otimes \k$,
where $K_0(\C^{\oplus})$ denotes the Grothendieck ring of the category $\C^{\oplus}$. 
\begin{pro}[Sliding property, \cite{BB}]
\label{slide}
In any pre-modular category the following handle sliding
identity holds
$$\psdiag{13}{28}{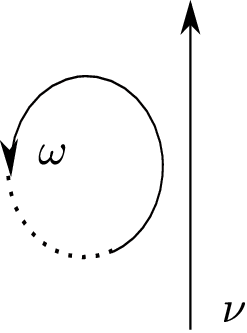}=\psdiag{13}{28}{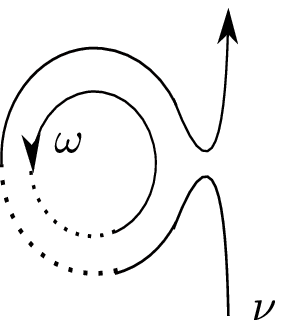}.$$
\end{pro}

%\AB{ I removed the proof, please refer to the graded proof later}

\begin{lem}[Killing property, \cite{BB}] Let $\C$ be a pre-modular category with 
$\langle\omega\rangle\neq 0$ and $\l\in \Gamma$. Then, the morphism 
\begin{center}
\includegraphics{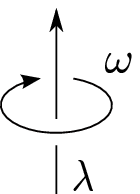}
\end{center}
is non-trivial if and only if $\l$ is transparent.
\end{lem}

\begin{proof}
If $\l$ is transparent, then the map
\begin{equation*}
 \psdiag{10}{20}{KP1} =\langle \omega \rangle \psdiag{10}{20}{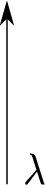}\neq 0.
\end{equation*}
Conversely, if we assume that
for some $c\in \k$, $c\neq 0$ we have
\begin{equation*}
 \psdiag{10}{20}{KP1} =c \psdiag{10}{20}{KP2}
\end{equation*}
%\begin{center}
%\includegraphics{KP3.eps}
%\end{center}
then, by sliding any $\nu$--colored strand along the
$\omega$--colored one, we obtain the following equalities of morphisms
in $\C$
\begin{equation*}
 \psdiag{10}{20}{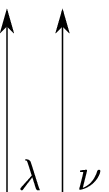} =c^{-1} \psdiag{10}{20}{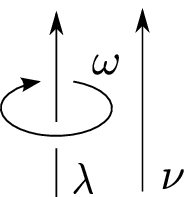} =\psdiag{10}{20}{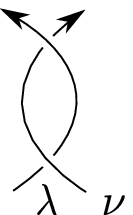}
\end{equation*}
%\begin{center}
%\includegraphics{KP5.eps}
%\end{center}
and therefore, $\lambda$ is transparent.
\end{proof}

\begin{pro}[\cite{BB}]
\label{twisted}
A pre-modular category with $\langle\omega\rangle\neq 0$
and with no non-trivial transparent objects is modular.
\end{pro}
\begin{proof}
Let $\overline{S}=(\overline{S}_{\l\nu})_{\l\nu}$ be a matrix with
entries given by the relation
\begin{equation*}
 \psdiag{10}{20}{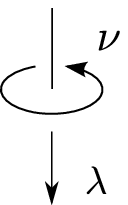}=\frac{\overline{S}_{\lambda\nu}}{\langle \lambda \rangle}\psdiag{10}{20}{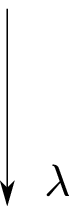}. 
\end{equation*}
\noindent
We want to prove that the product $S\cdot \overline{S}=\langle\omega\rangle \I$, 
where $\I$ is the identity matrix of size $|\Gamma|$.
We have
\begin{align*}
 (S\cdot \overline{S})_{\lambda \nu}&=\psdiag{10}{25}{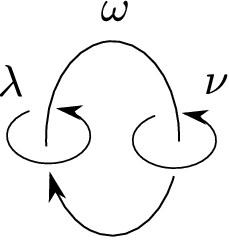}=\psdiag{10}{22}{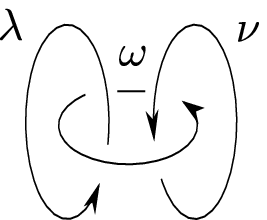}=\sum_{\mu \in \Gamma} \sum_{i\in I_{\lambda \nu^*}^{\mu}}\langle \mu \rangle \psdiag{10}{22}{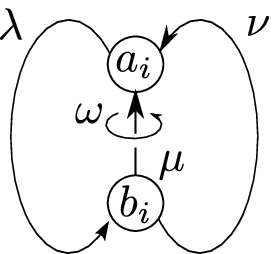}\\
 &=\langle \omega \rangle \sum_{i\in I_{\lambda \nu^*}^1} \psdiag{10}{24}{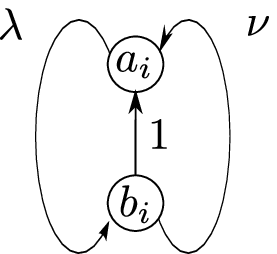}=\langle \omega \rangle \cdot \delta_{\lambda \nu},
\end{align*}
\noindent
where $\delta_{\l\nu}$ is the Kronecker index.
The second equality holds by isotopy, the third 
equality holds by the \emph{fusion formula}, while the fourth
equality is a consequence of the \emph{killing property}. 
The last equality can be proved
using the structure of the modules $\Hom(\1,\l\otimes \nu^*)$ and
$\Hom(\l\otimes \nu^*,\1)$.
\end{proof}
\begin{note}
 If $\C$ is a modular category, then $\langle \omega \rangle$ is invertible
 in $\k$, and hence $\langle \omega \rangle \neq 0$.
\end{note}

\subsection{WRT invariants}\label{WRT}
Given a modular category $\C$ and a closed $3$--manifold $M=S^3(L)$ obtained by surgery on
$S^3$ along a framed link $L$, whose linking matrix has
$b_+$ positive and $b_-$ negative eigenvalues, we define
\begin{equation}\label{WRT}
\tau_\C(M) =\frac{F_\C(L (\omega, \dots, \omega) )}{F_\C(U_{1}(\omega))^{b_+} F_\C(U_{-1}(\omega))^{b_-}}
\end{equation}
where $U_{\pm 1}$ denotes the $\pm 1$--framed unknot and
  $F_\C:\Rib_\C \to\C$ is the  natural  ribbon functor.

\begin{thm}[\cite{Tbook}]
\label{T}
For any modular category $\C$,
$\tau_\C(M)$ defines a topological invariant of $M$, independent
on the choice of the link $L$.
\end{thm}

\begin{proof} We need to show that
$\tau_\C(M)$ is well-defined and does not change under Kirby moves. 

The fact that $F_\C(U_{\pm 1}(\omega))$ are non-zero follows from the properties of the ribbon functor
$F_\C$ and the non--degeneracy axiom for $\C$. 

The first Kirby move is easy 
to establish. 
The invariance under
the second Kirby move
is provided by Proposition \ref{slide}. 
In order to prove invariance under orientation reversal,
let us consider the link $L'$ obtained from $L$ by reversing the 
orientation of a component $L_k$. Without  loss of generality,
we may assume $k=1$. Thanks to Corollary \ref{or} we have
\begin{align*}
 F_\C(L'(\omega,\cdots, \omega))&=\sum_{\l\in \Gamma}\langle \l \rangle F_\C(L'(\l,\omega,\cdots, \omega))=\sum_{\l\in \Gamma}\langle \l^* \rangle F_\C(L(\l^*,\omega,\cdots, \omega))\\
         &=\sum_{\mu \in \Gamma}\langle \mu \rangle F_\C(L(\mu,\omega,\cdots, \omega))=F_\C(L(\omega,\cdots, \omega)).
\end{align*}
The linking matrix $(L'_{ij})={^tS}(L_{ij})S$ with
$S=\diag(-1,1,\cdots,1)$ hence the matrices 
$(L'_{ij})$ and $(L_{ij})$ have the same eigenvalues. In particular,
$b'_+=b_+, b'_-=b_-$ so
%\begin{align*}
% F_\C(U_1(\omega))^{b'_+}&=F_\C(U_1(\omega))^{b_+}\\
 %F_\C(U_{-1}(\omega))^{b'_-}&=F_\C(U_{-1}(\omega))^{b_-}
%\end{align*}
$\tau_\C(M,L')=\tau_C(M,L)$. This concludes the proof.
\end{proof}

\section{Group--categories from invertible objects}
\label{graded}
In this section, we show that 
for any modular category $\C$,   its group of invertible  objects $G$ 
defines the structure of $\widehat{G}$--category on $\C$.
%acts on the set $\Gamma$ and any subgroup $H\subset G$ defines a grading by the group $\widehat{H}$ of its characters. 
%In this section
%we define refinable modular categories and spin modular categories.
We  identify an object with its isomorphism class and hence equality between objects
means an isomorphism.

\subsection{Invertible objects}
An object $\l$ in $\G$ is called {\em invertible} if there exists
another object $\nu\in \G$ such that $\l \otimes \nu= 1$.
Isomorphism classes of invertible simple objects form a finite abelian group 
under tensor multiplication.  Let us denote this group 
by $G\subset \Gamma$ and let
$\widehat{ G}=\Hom(G, \mathbb C^*)$ be the group of its characters.
We say that the \emph{tensor order}
of $\l$ is $d$ if $\l^{\otimes d}=1$ with $d$ minimal.

\begin{lem}
For any $\l\in \G$ and any invertible $g\in G$,
$\l\otimes g \in \G$. 
\end{lem}
\begin{proof}
By the domination axiom we have 
 $\l \otimes g=\oplus_i \;\mu_i$. Multiplying this identity by the inverse of $g$, we get
$\l=\oplus_i\;  (\mu_i \otimes g^{-1})$. Thus, the right hand is simple, and so
 $\l \otimes g$ has to be simple too.
\end{proof}
\begin{exa}
In Example \ref{mainex1}
 the object $K$ (the longest row)
is invertible of order $N$.
\label{mainex2}
\end{exa}

\subsection{Group--category}
Assume $\C$ is a modular category with $G$ the group
of invertible simple objects. We can define a $\widehat{G}$--structure on $\C$ as follows:
Given $\l\in \G$, the braiding coefficient of $\l$ with elements of $G$ defines a map
$\Gamma \to \widehat G$ which
associates to each $\l$ a character $\chi_\l
\in \widehat{G}$ defined by the equality:
$$\psdiag{10}{21}{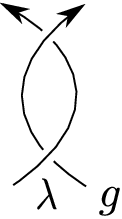}=\chi_{\lambda}(g) \psdiag{10}{20}{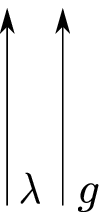}.$$
%\begin{center}
% \psfrag{a}{$\l$}
% \psfrag{b}{$g$}
% \psfrag{c}{$=\chi_{\l}(g)$}
% \includegraphics{grading.eps}
%\end{center}

Indeed, the previous lemma implies that
the braiding operator acts on
 $\l$ and $g$ as a multiplication by
 $S_{\l g}/\la\l\ra\la g\ra$, i.e. only one braiding coefficient is non-zero. 
Using  the fact that $g$ is of finite order, we deduce
that this coefficient is a root of unity of that order.
Observe that $\chi_{\lambda\otimes\mu}=\chi_\l\chi_\mu$  defines a group multiplication on $\widehat G$.
By taking a logarithm of the character, we can identify $G$ with its dual $\widehat G$ and write the
group operation on $G$ additively.

%Moreover, in this case $\la g\ra=\pm 1$, 
%since $\la g\ra^d=1$ and $\la g\ra$ is real (follows from
%$\la\l\ra=\la\l^*\ra, \forall \l\in \G$).

%For any $\a\in \widehat G$, we set $\C_\a=\{\l\in \Ob(\C) | \chi_\l=\a\}$.
Clearly,
 $\C$ splits into a disjoint union of subcategories $\C_\alpha$ for $\a \in \widehat G\cong G$.
Moreover,   for any $\l,\mu\in\G$,
$\Hom(\l,\mu)$ is either zero if
$\l\neq \mu$ or $\k$ otherwise  (by the assumption 
these objects are simple). Hence, 
we just proved the following.

\begin{pro}\label{Gcategory}
 A modular category $\C$ with a group $G$ of invertible objects  splits
as a disjoint union of subcategories $\{\C_\a\}_{\a\in  G}$ such that
\begin{itemize}
\item
each $\C_\a$ is a full subcategory of $\C$;
\item
each object of $\C$ belongs to $\C_\a$ for a unique $\a$;
\item
$\Hom_\C(\C_\a,\C_\beta)=0$ if $\a\neq \beta$;
\item
for $U\in \C_\a$ and $V\in \C_\beta$, $U \otimes V \in \C_{\a+\beta}$; 
\item
$1\in \C_0$, and for $U\in \C_\a$, $U^*\in \C_{-\a}$.
\end{itemize}
\end{pro}

Note that this is a special case of  group-categories defined by Turaev in
  \cite{THQFT}.

If $\lambda \in \C_\a$, then we will call $\a$ {\it degree} of $\l$ and denote by
%the degree of an object $\lambda$ with respect to the $\widehat{H}$--grading will be denoted 
by $\deg(\lambda)$.

If $G$ is cyclic of order $d$,  
we  call $\C$ a 
{\it modular $d$--category}.
Fixing a generator $t$ of $G$ and a primitive $d$th root of unity $e_d$,
 we have the decomposition
$$ \Gamma= \Gamma_0 \amalg \Gamma_1 \amalg \dots \amalg \Gamma_{d-1},$$
where $\Gamma_i=\{\l\in \Gamma| \chi_{\l}(t)=e_d^i\}$. 

%\begin{rem}
% A $\widehat{G}$-graded modular category is what Turaev called $G$-category in \cite{TuraevG}.
%\end{rem}

\begin{exa} Let us assume that $N=(N,K)$ in Examples \ref{mainex1} and \ref{mainex2}. Then $G=\Z_N$ 
is generated by $K$.
 Hence,
the category $\C(N,K)$ is  %$N$--graded 
a modular $N$-category with
$\deg(\l)= \sum_i \l_i \pmod N$. 
\end{exa}

We will need the following lemma.
\begin{lem}\label{lem21}
Let $\C$ be a 
%$d$--graded 
modular $d$--category such that the group $G=\langle t \rangle \cong \Z_d$.
Then  $\langle t\rangle=1$ if $d$ is odd, and $\langle t\rangle=\pm 1$ if $d$ is even. %Furthermore, if $\C$ is $d$--refinable then $\langle t\rangle=1$.
\end{lem}

\begin{proof} Let $b_{t,t}$ be the braiding coefficient, such that
\begin{center}
\psfrag{a}{$t$}
\psfrag{b}{$t$}
\psfrag{c}{$=\; b_{t,t}$}
\includegraphics{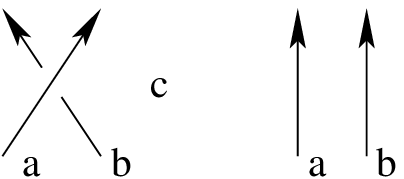}
\end{center}
Using
\begin{center}
\psfrag{a}{$t$}
\psfrag{b}{$=\; b_{t,t}$}
\psfrag{c}{$t$}
\psfrag{d}{$=\; b_{t,t}\langle t\rangle$}
\psfrag{e}{$=\theta_t$}
\includegraphics{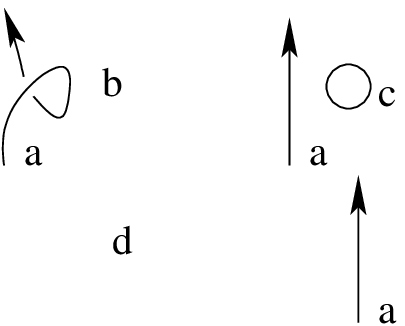}
\end{center}
we have  $\theta_t=b_{t,t}\la t\ra$. On the other hand, 
from Definition \ref{qtrace} we get
$\theta_t b_{t, t^*}=1$, where $t^*=t^{d-1}$, so $\langle t\rangle (b_{t,t})^d=1$. Since $(b_{t,t})^{2d}=1$, it follows that $\langle t\rangle=\pm 1$.
If $d$ id odd, then $\langle t^{\otimes d}\rangle=\la 1\ra=1=\langle t\rangle$.
\end{proof}
%Observe that the group $H\subset \G$ is itself graded. 
Let us recall the following definition which
will play a central role in our exposition.

%(part $\Gamma_0$ in the cyclic case, and trivial character in general). %We recall the definition 
%of a $G$-modular and spin $G$-modular category.

\begin{dfn}\label{Hspin} Let $H\subset G$ be a subgroup.
A modular category  with a group $G$ of invertible objects 
is called {\it  $H$--refinable} if
 $H\subset \C_0$.
%all cyclic subgroups of $H$ have  even order. 
%\end{itemize}
If, in addition, $H$ has at least one element with twist coefficient $-1$, 
we call the category {\it $H$--spin}.
 \end{dfn}

In the case $H$ is cyclic of order $d$, we will use the shorthand
$d$--refinable and $d$--spin.

\begin{exa} Given an even integer $d=(N,K)=\a\b$, such that $(\a, N')=(\b,K')=1$ for
$N'=\frac{N}{d}$ and $K'=\frac{K}{d}$,
The second author constructed in \cite{Blanchet} a category $\dot\C(N,K)$ with
$$\dot\Gamma=\{ (1^N)^{\otimes i} \otimes \l\; | \; 0\leq i< \a, \l_1\leq K, 
\l^\vee_1<N
\}$$
and $G=\Z_d$ generated by $1^N\otimes K$. Here $\l^\vee$ denotes the 
transpose partition.
%The object $1^N$ is invertible of odd order $\a$ and has twist coefficient $1$.
%The object $K$ is invertible of even order $\b$ and has twist coefficient $-1$.
The second author proved that the category $\dot\C(N,K)$ is $d$--spin modular.
%\AB{More details?? twist coeffs?}

\end{exa}

\subsection{Sliding identities in  pre-modular group-categories}
Assume $\C$ is a %$d$--graded 
pre-modular $d$--category.
For any $0\leq u<d$, let us define the refined Kirby colors
$\omega_u=\sum_{\l\in \G_u} \;\la\l\ra \l$
as objects in the additive closure of $\C$.

%We call a category $\C$ $d$--spin pre-modular if in Definition \ref{Hspin}
%$\C$ is pre-modular.

\begin{lem}[Graded sliding property]
 In any 
 %$d$--graded 
 pre-modular $d$--category $\C$ we have the following 
equality of morphisms:
 $$\psdiag{13}{28}{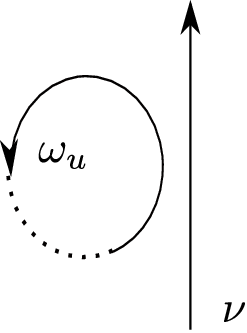}=\psdiag{13}{28}{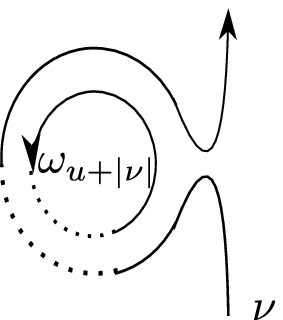}$$
 with $|\nu|=\deg(\nu)$.
\end{lem}

%\AB{the statement has nothing to do with beeing d-spin, no spin object (with twist -1 is required), so
%please correct is back and move this lemma to the previous section, since for the statement to hold the only thing we need is a $\widehat G$-grading}

\begin{proof}
The proof is the same as in the non-graded case, using
the fact that $\Hom(\l\otimes\nu, \mu)$ is zero
unless $\deg(\mu)=\deg(\l)+\deg(\nu)$. Hence, the sum over $\mu$ can be 
restricted to $\G_{u+|\nu|}$.
\begin{align*}
 \psdiag{13}{28}{slide11}&=\sum_{\lambda \in \Gamma_u} \langle \lambda \rangle \psdiag{13}{28}{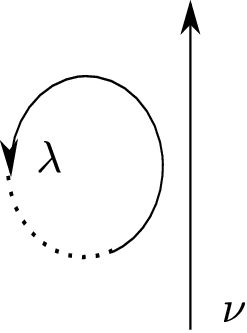}
 =\sum_{\lambda \in \Gamma_u}\sum_{\mu \in \Gamma_{u+l}}\sum_{i\in I_{\lambda \nu}^{\mu}}\langle \lambda \rangle \langle \mu \rangle \psdiag{13}{28}{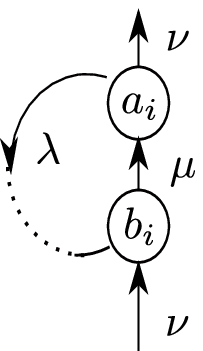}\\
 &=\sum_{\lambda\in \Gamma_u}\sum_{\mu\in \Gamma_{u+l}}\sum_{i\in I_{\lambda }^{\mu\nu^*}}\langle \lambda \rangle \langle \mu \rangle \psdiag{13}{28}{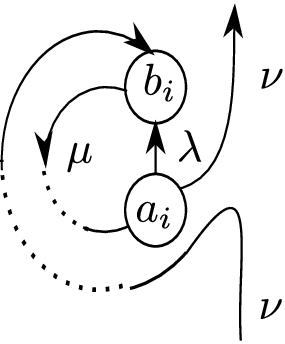}
 =\sum_{\mu\in \Gamma_{u+l}}\langle \mu \rangle \psdiag{13}{28}{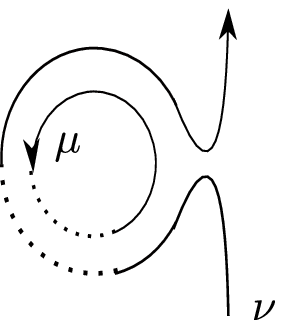}\\
 &=\psdiag{13}{28}{slide4}.
\end{align*} 

In the second and the fourth equalities we use the \emph{fusion formula}, the third equality holds by isotopy.

\end{proof}

 For any $0\leq v<d$, and a primitive $d$th root of unity $e_{d}$,
we define
$$\omega^v=\sum_{\l\in \G} e^{v\deg(\l)}_{d} \la\l\ra \, \l\, $$
the dual refined Kirby color. Note that $\omega^0=\omega$, however
other dual Kirby colors depend on the choice of $e_{d}$.
 We use the graded sliding identity to prove the following lemma.
%\AB{Please move also the dual sliding lemma to the section 5, here again only grading is used in the proof}

\begin{pro} [Dual sliding property]
In any %$d$--graded 
pre-modular $d$--category $\C$ we have the following equality of morphisms:
\begin{enumerate}[(a)]
 \item \centering $\psdiag{13}{28}{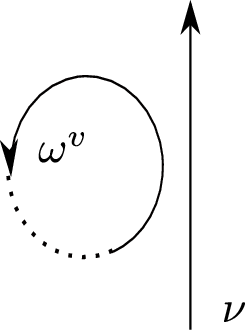}=e_{d}^{-v\deg(\nu)}\psdiag{13}{28}{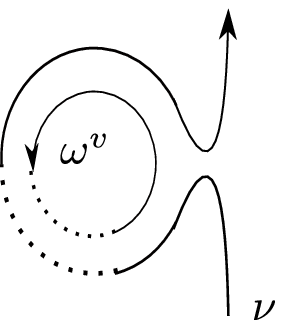}$
 \item  \centering $\psdiag{13}{28}{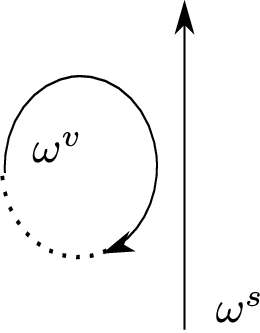}=\psdiag{13}{28}{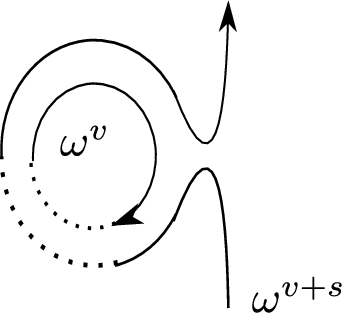}.$
\end{enumerate}
\end{pro}

\begin{proof}
(a) We have the following equalities:

%\AB{few factor are missing after the 2. equality }
\begin{align*}
 \psdiag{13}{28}{slide1d}&=\sum_{\lambda\in \Gamma} e_{d}^{v\deg(\lambda)} \langle\lambda\rangle \psdiag{13}{28}{slide111}
 =\sum_{\lambda \in \Gamma}\sum_{\mu\in \Gamma_{\deg(\lambda)+\deg(\nu)}}\sum_{i\in I_{\lambda\nu}^{\mu}}e_{d}^{v\deg(\lambda)} \langle\lambda\rangle \langle\mu\rangle\psdiag{13}{28}{slide2}\\
 &=e_{d}^{-v\deg(\nu)}\sum_{\mu\in \Gamma}\sum_{\lambda\in \Gamma_{\deg(\mu)-\deg(\nu)}}e_{d}^{v\deg(\mu)}\langle\lambda\rangle\langle\mu\rangle
 \psdiag{13}{28}{slide3}\\
 &=e_{d}^{-v\deg(\nu)}\sum_{\mu\in \Gamma}e_{d}^{v\deg(\mu)}\langle\mu\rangle \psdiag{13}{28}{slide41}
 =e_{d}^{-v\deg(\nu)}\psdiag{13}{28}{slide42d}.
\end{align*}
\noindent
The second and fourth equalities follow from the \emph{fusion formula}, the third equality holds by isotopy.\\
(b) According to Corollary \ref{or}, given a $\C$--colored ribbon graph $\Omega$ with an annulus component
$a$ colored with the dual Kirby color $\omega^s$, if we consider the ribbon
graph $\Omega'$ obtained from $\Omega$ by reversing the orientation of
$a$ and changing its color to $\omega^{-s}$, then $F_\C(\Omega)= F_\C(\Omega')$.
In particular, we have
\begin{equation*}
 \psdiag{13}{28}{slide1d1}=\psdiag{13}{28}{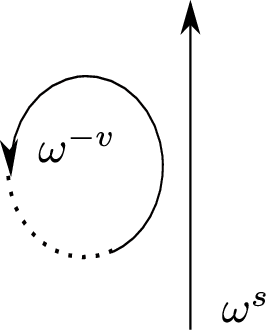}=\psdiag{13}{28}{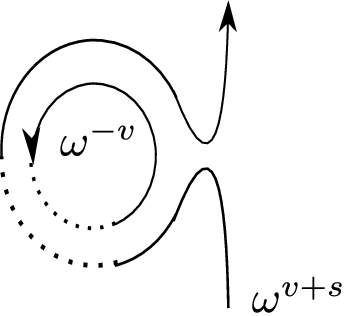}=\psdiag{13}{28}{slide42d1},
\end{equation*}

\noindent where the second equality is a straightforward application of (a).
\end{proof}

%\begin{rem}
%From the proof we see that any degree zero object $t\in G\cap \Gamma_0$ with $\theta_t=-1$
%has to be of even order. 
%\end{rem}

\section{Spin modular categories}
\label{refined}
This section is devoted to the proof of Theorem 1. 
Motivated by the known examples we  will first give the
proof  in the cyclic case.
For $d$--spin modular categories 
we obtain invariants of $3$--manifolds equipped with $\Spin_d$ structure.
In cohomological case we get refined invariants for $3$--manifolds with 
modulo $d$ $1$--dimensional cohomology classes.

\noindent For the rest of this section 
$M=S^3(L)$ is a closed $3$--manifold obtained
by surgery on $S^3$ along a framed link $L=(L_1,\cdots, L_n)$ whose linking matrix $(L_{ij})$
has $b_{+}$ positive and $b_{-}$ negative eigenvalues.

\subsection{Spin refinements, cyclic case}
Let $d$ be an even, positive integer and $\C$ be a $d$--spin modular category.
For any $s=(s_1,\dots, s_n)\in \Spin_d(M)$ let us define
\begin{equation}\label{WRT-spin}
\tau_\C(M, s) =\frac{F_\C(L (\omega_{s_1}, \dots, \omega_{s_n}) )}{F_\C(U_{1}(\omega))^{b_+} F_\C(U_{-1}(\omega))^{b_-}}.
\end{equation}

\begin{thm}\label{main-spin}
For any $d$--spin modular category $\C$,
$\tau_\C(M, s)$ is a topological invariant of the pair $(M, s)$.
Moreover,
$$\tau_\C(M)=\sum_{s \in \Spin_d(M)} \tau_\C(M,s)\, .$$
\end{thm}
%\AB{I would remove L from the definition, if you like you can introduce it back in the proof}

\begin{proof}
To prove the first statement we need to show that $\tau{(M, s)}$ is well-defined
and does not change under refined Kirby moves.

Invariance under the first two
 $\Spin_d$ Kirby moves follows immediately from
the \emph{graded sliding property} and the next lemma
which  implies that $F_\C(U_\epsilon(w))= F_\C(U_\epsilon (w_{\frac{d}{2}}))$
is invertible for $\epsilon=\pm 1$. In order to prove invariance under
orientation reversal, let us consider the link $L'$ obtained from $L$
by reversing the orientation of a component $L_k$. Without loss of 
generality, we may assume that $k=1$. We have seen in the proof of Theorem
\ref{T} that the linking matrices $(L'_{ij})$ and $(L_{ij})$ have the same
eigenvalues so $b'_+=b_+$, $b'_-=b_-$.\\
Since $s'=(-s_1,s_2,\cdots, s_n)$ applying Corollary \ref{or} gives
\begin{align*}
 F_\C(L'(\omega_{-s_1}, \omega_{s_2},\cdots,\omega_{s_n}))&=\sum_{\l\in \G_{-s_1}}\langle \l \rangle F_\C(L'(\l,\omega_{s_2},\cdots,\omega_{s_n}))\\
 &=\sum_{\l\in \G_{-s_1}}\langle \l^* \rangle F_\C(L(\l^*,\omega_{s_2},\cdots,\omega_{s_n}))\\
 &=\sum_{\mu \in \G_{s_1}}\langle \mu \rangle F_\C(L(\mu,\omega_{s_2},\cdots,\omega_{s_n}))\\
 &=F_\C(L(\omega_{s_1}, \omega_{s_2},\cdots,\omega_{s_n}))
\end{align*}
and therefore $\tau_\C(M,L';s')=\tau_\C(M,L;s)$.
It remains to prove that $F_\C(L (\omega_{s_1}, \dots, \omega_{s_n}) )=0$
if $s_i$ do not solve $\sum_{j=1}^{n} L_{ij} s_j=\frac{d}{2} L_{ii} \pmod d$.
The proof is in three steps.

Assume that the first component  $L_1$ of our link
is the $\pm 1$--framed unknot. Then it can be unlinked from the rest of $L$
by applying the Fenn-Rourke move. The graded sliding identity and
Lemma \ref{cancel}
tell us that $s_1$ should solve the above equations.

Assume
$L_1$ is the $a$--framed unknot. Then we add a $\pm 1$--framed unknot
to our link (with an invertible invariant) and slide it along $L_1$
(perform the inverse Fenn-Rourke move). This
 changes the framing on $L_1$ by $\mp 1$ and
allows to reduce this case to the previous one.

Finally, assume $L_1$ is arbitrary. Then we can unknot $L_1$ by adding
$\pm 1$--framed unknots to our link in such a way, that their linking number 
with $L_1$ is zero. This again reduces the situation to the previous case.
\end{proof}

\begin{lem}\label{cancel} For any $d$--spin  modular category $\C$, 
$F_\C(U_{\pm 1}(w_u)) $ is zero unless $u=\frac{d}{2}$.
\end{lem}

\begin{proof} 
%Assume that $t_k$ has the twist coefficient $-1$ and even tensor order $l_k$
%if $k\in I\subset T$.

Recall that invertible objects form an abelian group under
tensor multiplication, which acts on $\G$. In particular,
its cyclic subgroup $H=\la t\ra \cong \Z_d$ acts on each $\G_u$. Let us denote by
$\widetilde \G_u$ the set of orbits under this action and by
$\widetilde \omega_u$ the corresponding reduced Kirby color. Note that
$\omega_u=\sum_{i=0}^{d-1}\langle t\rangle^i t^i \widetilde \omega_u$.
Let $H_{a,b}$ be the $(a,b)$--framed Hopf link with linking
matrix $-1$.
\begin{figure}[h!]
\begin{center}
\psfrag{a}{$\omega_u$}
\psfrag{b}{$\omega_0$}
\includegraphics{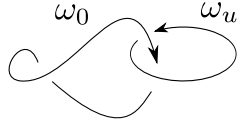}
\caption{The Hopf link $H_{1,0}(\omega_0,\omega_u)$.}
\end{center}
\end{figure}

After sliding the second component of $H_{(1,0)}(\omega_0,\omega_u)$ along the first one
we get
\begin{align*} 
 F_\C(U_1(w_u)) F_\C(U_{-1}(w_u))& =
F_\C(H_{1,0}(\omega_0, \omega_u))\\ & =\sum_{i,j=0}^{d-1}\langle t\rangle^{i+j}
F_\C(H_{1,0}(t^i\widetilde \omega_0,t^j \widetilde \omega_u)) \\ &=
\sum_{i,j=0}^{d-1}\langle t\rangle^{i+j} 
F_\C(H_{1,0}(\widetilde \omega_0,\widetilde \omega_u))
e^{u i}_{d} \theta_t^i \langle t\rangle^{i+j} \\ &= dF_\C(H_{1,0}(\widetilde \omega_0,\widetilde \omega_u))
\sum_{i=0}^{d-1}e^{u i}_{d}(-1)^i
\end{align*} 
which is zero unless $u=\frac{d}{2} \pmod{d}$.
% since $(1+e_d^u)\sum^{d-1}_{i=0} (-1)^i e^{u i}_{d} =0$.
Here we used that 
$\theta_t=-1$ and $\la t\ra=\pm 1$, by Lemma \ref{lem21}. Since
 $F_\C(U_1(w_u))$ and $ F_\C(U_{-1}(w_u))$ are complex conjugate to each other,
the result follows.

%\AB{Add the argument why each factor is non-zero unless u=d/2}

%here we use the fact that in a spin d-graded category
% the twist coefficient of the generator invertible t
%is -1
\end{proof}

\subsection{Cohomological refinements}
In this subsection we 
assume that $\C$ is a non-spin $d$--refinable modular category.
%We will call such $\C$ non-spin $\Z_d$--modular category.

%\AB{perhaps we have to change refinable in the intro by non-spin??
%or just G-modular category?? like in turaev's book?? }

The elements $h\in H^1(M; \Z_d)$ are combinatorially given
by solutions of $\sum_{j=1}^{n} L_{ij} h_j=0 \pmod d, \forall i=1,...,n.$
Let $L$ be an oriented framed link and $h\in H^1(S^3(L);\Z_d)$. The usual Kirby moves admit 
refinements for manifolds equipped with such structures as follows:
\begin{itemize}
 \item Stabilization: $(L,h) \longrightarrow (L\sqcup U_{\pm 1}, (h,0))$;
 \item Handle slide: $(L,h)\longrightarrow (L',h')$ where $L'$ is obtained from $L$
 by sliding component $L_i$ along $L_j$ and $h'_k=h_k$ if $k\neq j$ and
 $h'_j=h_j\mp h_i$. Here the sign depends on whether the orientations of $L_i$ and $L_j$ match or not, respectively;
 \item Orientation reversal: $(L,h)\longrightarrow (L',h')$ where $L'$ is obtained from $L$ by reversing the orientation 
 of component $L_i$ and $h'_j=h_j$ if $j\neq i$ and $h'_i=-h_i$.
\end{itemize}

For any $h=(h_1,\dots, h_n)\in H^1(M;\Z_d)$
let us define
\begin{equation}\label{WRT-coho}
\tau_\C(M,h) =\frac{F_\C(L (\omega_{h_1}, \dots, \omega_{h_n}) )}{F_\C(U_{1}(\omega))^{b_+} F_\C(U_{-1}(\omega))^{b_-}}.
\end{equation}
%where we denote by $U_\varepsilon$ the $\varepsilon$--framed unknot.

\begin{thm}\label{main-coho}
For any non-spin $d$--refinable modular category,
$\tau_\C(M, h)$ is a topological invariant of the pair $(M, h)$.
Moreover,
$$\tau_\C(M)=\sum_{h \in H^1(M;\Z_d)} \tau_\C(M,h)\, .$$
\end{thm}

\noindent
The proof
is based on the following lemma.

\begin{lem} 
\label{vanish}
For any non-spin $d$--refinable modular category $\C$, 
$F_\C(U_{\pm 1}(w_u)) $ is zero unless $u=0$.
\end{lem}

\begin{proof}
 Just as in the proof of Lemma \ref{cancel}, we consider 
 the Hopf link $H_{1,0}(\omega_0, \omega_u)$ with linking 
 number $-1$ and we slide the second component
 along the first one. 
 We get 
  \begin{align*} 
 F_\C(U_1(w_u)) F_\C(U_{-1}(w_u))& =
F_\C(H_{1,0}(\omega_0, \omega_u))\\ & =d
F_\C(H_{1,0}(\widetilde \omega_0, \widetilde \omega_u)) 
 \sum^{d-1}_{j=0} 
 e^{u j}_{d} 
\end{align*}
which is zero, unless $u=0$.
%\AB{discuss Kirby moves for the pair (M,h) somewhere, probably at the beginning}

%here we use the fact that in a non-spin d-graded category,
%the twist coefficient of the generator invertible t is 
%+1
\end{proof}
%\begin{figure}[h!]
%\begin{center}
% \includegraphics{Kirby3.eps}
%\end{center}
%\caption{Refined cohomological Kirby moves. (a) Stabilization
%(b) Handle slide (c) Orientation reversal. The blackboard framing is used
%and the labels refer to elements
%$h\in H^1(M;\Z_d)$.\label{cohomological}}
%\end{figure}

\subsection{Spin refinements, general case}
The proof of Theorem 1 follows from the two previous cases.
 An $(\widehat{H},v)$ generalized spin structure $\sigma$ on $M=S^3(L)$
 is described by a sequence of coefficients $(s_1,\dots,s_n)\in \widehat H^n$ satisfying a characteristic equation:
 $$(s_1,\dots,s_n)(\widehat H\otimes L)=v\otimes (L_{11},\dots,L_{nn})\ $$
Indeed, the Kirby element decomposes using $\widehat H$-grading, and the formula for the refined invariant given in \eqref{WRT-spin} still holds.
The condition for non vanishing in Lemma \ref{cancel} is $u=v$.

The solution of these equations is a sequence of $\sigma_i$ where the index $i$ runs over 
cyclic components  $H_i\cong \Z_{d_i}$ of $H$. Moreover, either $\sigma_i\in \Spin_{d_i}(M)$ or $\sigma_i\in H^1(M,\Z_{d_i})$,
depending on the twist coefficient of the corresponding generator. Combining the two previous theorems we get
$$ \tau_\C (M)=\sum_\sigma \tau_\C(M,\sigma).$$

\section{
Complex spin refinements}
This section is devoted to the proof of Theorem 2. 
Again, according to the twist coefficients, we will either get
 an extension of WRT invariants for $3$--manifolds equipped
 with modulo $d$ complex spin structures or with $2$--dimensional cohomology classes.

Throughout this section $\C$ is a $2d$--spin modular category with
$d$ even.
\noindent For any $\sigma \in \mathcal{S}^c_d(M)$, let us define
\begin{equation}\label{WRT-spinc}
\tau_\C(M, \sigma) =(-d)^{-n}\sum_{(\epsilon_1,...,\epsilon_n)\in \ \sigma}
\frac{F_\C(L (\omega^{\epsilon_1}, \dots, \omega^{\epsilon_n}) )}
{F_\C(U_{1}(\omega))^{b_+} F_\C(U_{-1}(\omega))^{b_-}}
\end{equation}
where the shorthand $(\epsilon_1,...,\epsilon_n)\in \sigma$ means that the summation is taken over all
elements of $(\Z_{2d})^n$ in the equivalence class of $\sigma$.

%\AB{Define the range of epsilons!}

\begin{thm}\label{main-spinc} Let $d$ be an even integer.
For any $(2d)$--spin modular category $\C$, 
$\tau_\C(M, \sigma)$ is a topological invariant of the pair $(M, \sigma)$.
\end{thm}

%\AB{2d refinable category will lead here to refinements wrt to $H^2(M, Z_{2d})$
%we have to add the combinatorial description of such structures}

\begin{proof}
In order to prove that $\tau_\C(M, \sigma)$ is a topological invariant
of the $\Spin^c_d$--manifold $(M,\sigma)$,
we have to check invariance under the $\Spin_d^c$ Kirby moves of Theorem \ref{Kirby}.

We start by checking invariance under the first Kirby move.
Let $(L',\sigma')$ be obtained from $(L, \sigma)$ by a positive stabilization. We have that
\begin{eqnarray*}
\tau_\C(M,L';\sigma')&=&(-d)^{-(n+1)}\sum_{(\epsilon_1,...,\epsilon_{n+1})\in \sigma'}
\frac{F_\C(L (\omega^{\epsilon_1}, \dots, \omega^{\epsilon_n}) )F_\C(U_1(\omega^{\epsilon_{n+1}}))}
{F_\C(U_{1}(\omega))^{b_+ +1} F_\C(U_{-1}(\omega))^{b_-}}\\
&=&(-d)^{-n}\sum_{(\epsilon_1,...,\epsilon_n)\in \sigma}
\frac{F_\C(L (\omega^{\epsilon_1}, \dots, \omega^{\epsilon_n}) )}
{ F_\C(U_{1}(\omega))^{b_+} F_\C(U_{-1}(\omega))^{b_-}} \sum_{\epsilon_{n+1}\in 1+2 \mathbb{Z}_{2d}}
\frac{F_\C(U_1(\omega^{\epsilon_{n+1}}))}{(-d)F_\C(U_{1}(\omega))}\\
&=&\tau_\C(M,L;\sigma) \sum_{\substack{x=0\\
       x \  \text{\ even}}}^{2d-1}
\frac{F_\C(U_1(\omega^{1+x}))}{(-d)F_\C(U_1(\omega))}.
\end{eqnarray*}

To compute the sum $\sum_{\substack{x=0\\ x\ \text{even}}}^{2d-1}F_\C(U_1(\omega^{1+x}))$,
we write the dual Kirby color $\omega^{1+x}$ in terms of the refined (graded) Kirby colors $\omega_i$
as follows:
\begin{eqnarray*}
 \omega^{1+x}&=&\sum_{i=0}^{2d-1}\sum_{\l\in \Gamma_i} e_{2d}^{(1+x)i}\la \l \ra \l =\sum_{i=0}^{2d-1}e_{2d}^{(1+x)i}\omega_i.
\end{eqnarray*}

Lemma \ref{cancel} together with the identity
$ \sum_{\substack{x=0\\ x \ \text{even}}}^{2d-1}e_{2d}^{(1+x)i} =
\left\{
	\begin{array}{ll}
		0  & \mbox{if } i \neq d \\
		-d & \mbox{if }  i=d
	\end{array}
\right.
$
gives
\begin{eqnarray*}\sum_{\substack{x=0\\
                 x\ \text{even}}}^{2d-1}F_\C(U_1(\omega^{1+x}))&=&
\sum_{i=0}^{2d-1}\sum_{\substack{
x=0\\
x \ \text{even}}}^{2d-1}e_{2d}^{(1+x)i}F_\C(U_1(\omega_i))
\\
&=& -d  F_\C(U_1(\omega))
%e_d^{-v|\n|}\sum_{\m\in {\Gamma}}e_d^{v|\m|}\la\m\ra\psdiag{8}{24}{sl4}\\
%&=&
\end{eqnarray*}
so $\tau_{\C}(M,L';\sigma')=\tau_{\C}(M,L;\sigma)$.

Analogously, $\tau_{\C}(M, \sigma)$ is invariant under a negative stabilization.
The invariance under the second Kirby move is provided by the \emph{dual sliding property}.
Finally, we must check invariance under orientation reversal. For that, let $(L', \sigma')$ 
be obtained from $(L,\sigma)$
by changing the orientation of a component $L_k$.
Without any loss of generality, we may assume that $k=1$ and, just like in the proof of Theorem \ref{T},
we get $b'_+=b_+$, $b'_-=b_-$. We have that
\begin{align*}
 \sum_{(\epsilon'_1,\cdots,\epsilon'_n)\in \sigma'}F_\C(L'(\omega^{\epsilon'_1},\cdots, \omega^{\epsilon'_n}))&=\sum_{(\epsilon_1,\cdots,\epsilon_n)\in \sigma}F_\C(L'(\omega^{-\epsilon_1},\omega^{\epsilon_2}, \cdots, \omega^{\epsilon_n}))\\
&=\sum_{(\epsilon_1,\cdots,\epsilon_n)\in \sigma} \sum_{\l\in \Gamma}e_{2d}^{-\epsilon_1 \deg(\l)}\langle \l \rangle F_\C(L'(\l,\omega^{\epsilon_2}, \cdots, \omega^{\epsilon_n}))\\
&=\sum_{(\epsilon_1,\cdots,\epsilon_n)\in \sigma} \sum_{\l\in \Gamma}e_{2d}^{\epsilon_1 \deg(\l^*)}\langle \l^* \rangle F_\C(L(\l^*,\omega^{\epsilon_2}, \cdots, \omega^{\epsilon_n}))\\
&=\sum_{(\epsilon_1,\cdots,\epsilon_n)\in \sigma} \sum_{\mu \in \Gamma}e_{2d}^{\epsilon_1 \deg(\mu)}\langle \mu \rangle F_\C(L(\mu,\omega^{\epsilon_2}, \cdots, \omega^{\epsilon_n}))\\ 
&=\sum_{(\epsilon_1,\cdots,\epsilon_n)\in \sigma} F_\C(L(\omega^{\epsilon_1},\cdots,\omega^{\epsilon_n})).
 \end{align*}
 The first equality above is a consequence of the fact that $(L'_{ij})={^tS}(L_{ij})S$, for $S=\diag(-1,1,\cdots,1)$ while the 
 third equality is an immediate application of Corollary \ref{or}. This concludes the proof.
\end{proof}

\subsection{Homological refinements} Let $d$ be a positive integer and
$\mathcal{C}$ be a non--spin $d$--refinable modular category.

%Let $M=S^3(L)$ be a closed $3$--manifold obtained by surgery along an $n$-component framed 
%link $L$, whose linking matrix $(L_{ij})_{ij}$ has $b_+$ (resp. $b_-$) positive (resp. negative)
%eigenvalues. 
The group $H_1(M;\Z_d)$ is described combinatorially as the set $(\mathbb{Z}_d)^n/\Im L$.
The  Kirby moves for the pair $(M, h)$ where $M$ is obtained by surgery on a link $L$ and
$h\in H_1(M;\Z_d)$ can be described as follows:
\begin{itemize}
 \item Stabilization: $(L,h)\longrightarrow (L\sqcup U_{\pm 1}, (h,0))$;
 \item Handle slide: $(L,h)\longrightarrow (L',h')$ where $L'$ is obtained from $L$
 by sliding component $L_i$ along $L_j$ and $h'_k=h_k$ if $k\neq i$ and $h'_i=h_i\pm h_j$.
 Here the sign depends on whether the orientations of $L_i$ and $L_j$ match or not, respectively;
 \item Orientation reversal: $(L,h)\longrightarrow (L',h')$ where $L'$ is obtained from $L$ by changing
 the orientation of component $L_i$ and $h'_j=h_j$ if $j\neq i$ and $h'_i=-h_i$.
\end{itemize}

For any $h \in H_1(M;\mathbb{Z}_d)$ let us  define
\begin{equation}
 \tau_\C(M, h) =d^{-n}\sum_{(\epsilon_1,...,\epsilon_n)\in \ h}
\frac{F_\C(L (\omega^{\epsilon_1}, \dots, \omega^{\epsilon_n}) )}
{ F_\C(U_{1}(\omega))^{b_+} F_\C(U_{-1}(\omega))^{b_-}}
\end{equation}
where the shorthand $(\epsilon_1,...,\epsilon_n)\in h$ means that the summation is taken over all
elements of $(\Z_d)^n$ in the equivalence class of $h$.

\begin{thm}\label{main-cohom2} 
For any non--spin $d$--refinable modular category $\C$, 
$\tau_\C(M, h)$ is a topological invariant of the pair $(M, h)$. 
\end{thm}

\begin{proof}
In order to prove that $\tau_\C(M, h)$ is a topological invariant
of the manifold $(M,h)$,
we have to check invariance under the Kirby moves listed above.

We start by checking invariance under the first Kirby move.
Let $(L',h')$ be obtained from $(L,h)$ by a positive stabilization. 
\begin{eqnarray*}
\tau_\C(M,L';h')&=&d^{-(n+1)}\sum_{(\epsilon_1,...,\epsilon_{n+1})\in h'}
\frac{F_\C(L (\omega^{\epsilon_1}, \dots, \omega^{\epsilon_n}) )F_\C(U_1(\omega^{\epsilon_{n+1}}))}
{ F_\C(U_{1}(\omega))^{b_+ +1} F_\C(U_{-1}(\omega))^{b_-}}\\
&=&d^{-n}\sum_{(\epsilon_1,...,\epsilon_n)\in h}
\frac{F_\C(L (\omega^{\epsilon_1}, \dots, \omega^{\epsilon_n}) )}
{ F_\C(U_{1}(\omega))^{b_+} F_\C(U_{-1}(\omega))^{b_-}} \sum_{i=0}^{d-1}
\frac{F_\C(U_1(\omega^i))}{d F_\C(U_1(\omega))}\\
&=&\tau_\C (M,L;h) \frac{F_\C(U_1(\sum_{i=0}^{d-1}\omega^i))}{d F_\C(U_1(\omega))}.
\end{eqnarray*}
We compute \begin{eqnarray*}
 \sum_{i=0}^{d-1} \omega^i&=&\sum_{i=0}^{d-1}\sum_{j=0}^{d-1}\sum_{\lambda \in \G_j} e_d^{ij} \la \lambda \ra \lambda\\
 &=&\sum_{j=0}^{d-1}\sum_{\lambda \in \G_j}(\sum_{i=0}^{d-1}e_d^{ij})\la \lambda \ra \lambda\\
 &=& d \cdot \omega_0
\end{eqnarray*}
since $\sum_{i=0}^{d-1}e_{d}^{ij}=0$ unless $j=0$. Using Lemma \ref{vanish} we get that $F_\C(U_1(\omega))=F_\C(U_1(\omega_0))$
and therefore $\tau_\C(M,L';h')=\tau_\C(M,L;h)$.

Analogously, $\tau_{\C}(M, h)$ is invariant under a negative stabilization.
The invariance under the second Kirby move is provided by the \emph{dual sliding property}.
Finally, we must check invariance under orientation reversal. For that, let $(L',h')$ 
be obtained from $(L,h)$
by changing the orientation of a component $L_k$.
Without any loss of generality, we may assume that $k=1$ and, just like in the proof of Theorem \ref{T},
we get $b'_+=b_+$, $b'_-=b_-$.
Since $h'=(-h_1,h_2,\cdots,h_n)$, it follows that
\begin{align*}
 \sum_{(\epsilon'_1,\cdots,\epsilon'_n)\in h'}F_\C(L'(\omega^{\epsilon'_1},\cdots, \omega^{\epsilon'_n}))&=\sum_{(\epsilon_1,\cdots,\epsilon_n)\in h}F_\C(L'(\omega^{-\epsilon_1},\omega^{\epsilon_2}, \cdots, \omega^{\epsilon_n}))\\
&=\sum_{(\epsilon_1,\cdots,\epsilon_n)\in h} \sum_{\l\in \G}e_d^{-\epsilon_1 \deg(\l)}\langle \l \rangle F_\C(L'(\l,\omega^{\epsilon_2}, \cdots, \omega^{\epsilon_n}))\\
&=\sum_{(\epsilon_1,\cdots,\epsilon_n)\in h} \sum_{\l\in \G}e_d^{\epsilon_1 \deg(\l^*)}\langle \l^* \rangle F_\C(L(\l^*,\omega^{\epsilon_2}, \cdots, \omega^{\epsilon_n}))\\
&=\sum_{(\epsilon_1,\cdots,\epsilon_n)\in h} \sum_{\mu \in \G}e_d^{\epsilon_1 \deg(\mu)}\langle \mu \rangle F_\C(L(\mu,\omega^{\epsilon_2}, \cdots, \omega^{\epsilon_n}))\\ 
&=\sum_{(\epsilon_1,\cdots,\epsilon_n)\in h} F_\C(L(\omega^{\epsilon_1},\cdots,\omega^{\epsilon_n})).
 \end{align*}
 The first equality above is a consequence of the fact that $(L'_{ij})={^tS}(L_{ij})S$, for $S=\diag(-1,1,\cdots,1)$ while the 
 third equality is an immediate application of Corollary \ref{or}. This concludes the proof.
\end{proof}

%In the 
%non-cyclic case, grading is a multivector in cyclic summands, 
%and for each component of this vector, we can speak about 
%spin/non-spin or not nice  gradings.

%(Since this generalization is straightforward, but notation consuming)
%We will assume $G=\Z_d$ in what follows.

\section{Decomposition formula}
This section is devoted to the proof of Theorem 3. Reader interested in the case $\gcd(m, \delta)=1$
only can skip this section and consult
an easy direct argument in Appendix.

Throughout this section $d$ is any positive integer, it needs not to be even.

\subsection{Strategy of the proof}
%We now consider the case of a modular category with general cyclic group of invertible objects.
Let us recall the setting.
We assume that $\C$ be a  modular category with cyclic group of invertible objects
 $G=\la t\ra$.
Let $\deg(t)=\delta\in \Z_d\cong \widehat G$ and $m\delta=d$.
Moreover, let us split
 $\delta=\alpha \beta$, such that $\gcd(\beta, \alpha m)=1$,
$\alpha\equiv m$ $\pmod 2$.
%Then $\C$ is $\delta$--refinable, moreover
%in cohomological case (resp. in spin case) there exists a  reduced $\delta$--refinable category   
% $\widetilde{\C}$ and a root of unity $\eta$ such that
%for any closed oriented $3$-manifold $M$ and any $\sigma \in H^1(M,\Z_\delta)$ (resp. $\sigma \in \Spin_\delta(M)$ in $\delta$--spin case) we have
% $$\tau_\C(M,\sigma)=m^{-b_1(M)}\tau_{\widetilde{\C}}(M,\sigma) \tau^{\xi}_{\MOO}(M,\sigma) $$
% where $\tau_\C(M,\sigma)$ and $\tau_{\widetilde{\C}}(M,\sigma)$   are the  WRT refined invariants,  $b_1(M)$ is the first Betti number, and 
%$\tau^{\xi}_\MOO(M,\sigma)$  is a Murakami--Ohtsuki--Okada \cite{MOO} type invariant associated with the root of unity $\xi$.
%\end{thm}

%Here $\beta : H^1(M,\Z_\delta)\rightarrow H^2(M,\Z/\alpha m$ (resp.  $\Spin_\delta(M)\rightarrow (\Spin_\delta mM,\Z/\alpha m$
%The correction term is evaluation of  $\lambda(\beta(\sigma))}$ is the 
%If $\delta$ is even, then $\widetilde \C$ is $\delta$--refinable.

%We first fix some notation and show that the modular category $\mathcal{C}$ is $\delta$-refinable.
%We denote by $F: \Rib_{\mathcal{C} }\rightarrow \mathcal{C}$ the Reshetikhin-Turaev ribbon functor.

%We have \mbox{$F(\text{identity of $V$ with $t$-colored meridian})=\kappa^{\deg(V)} \langle t\rangle Id_V$},

For any $V\in \Ob(\C)$ we can choose $\chi_V(t)=\kappa^{\deg(V)}$
where $\kappa$ is a primitive $d$--th root of unity.
 Let us fix the generator $t$, so that $\kappa=e^{\frac{i2\pi}{d}}$. 
The twist coefficient for $t$ is $\theta_t=b_{t,t}\langle t\rangle$ and satisfies
$\theta_t^2=\kappa^\delta=e^{\frac{i2\pi}{m}}$.
We consider the subgroup of invertible objects $H=\langle t^m\rangle\cong \Z_\delta$. Clearly, $H\subset \C_0$,
%It is included in trivial $\widehat{H}$-degree, 
so the modular category $\mathcal{C}$ is $\delta$--refinable. The twist coefficient for the generator
$t^m$ is $\theta_{t^m}=b_{t,t}^{m^2}\langle t\rangle^m$. It is equal to $-1$ if
$\delta$ is even, $m$ is odd and either $\langle t\rangle=1$ and $b_{t,t}=\theta_t=e^{\frac{i\pi}{m}}$, or $\langle t\rangle=-1$ and $b_{t,t}=-\theta_t=-e^{\frac{i\pi}{m}}$; it is equal to $1$ in all other cases.
 The modular category $\mathcal{C}$ is $\delta$--spin if $\delta$ is even, 
$m$ is odd and $\theta_t$ has order $2m$, and $\delta$--cohomological in all other cases.

We now present the idea of the proof of the decomposition statement.
% Characters on the cyclic group $H=<t>\simeq \Z/d$
%are identified with  $\Z_d$. Let $\delta\in \Z_d\cong \widehat H$ be the degree of $t$ and $m=\frac{d}{\delta}$.
%We write $\delta=\alpha\beta$ with $(\beta,\alpha m)=1$ and $\alpha\equiv m$ mod. $2$.
We define a tensor category ${\mathcal{C}'}$ with simple objects represented by 
${\G'}=\G\times \Z_\alpha$.
%, grading group $\Z/\alpha d$.
The tensor product in ${\mathcal{C}'}$ mimics central extension
of groups using $2$-cocycles. We lift the map  $\deg: \Gamma\rightarrow \Z_d$ into a map $f: \Gamma\rightarrow \Z_{d\alpha}$, which plays the role of a section.
%An object in $\mathcal{C}'$ is a finite sequence of elements of $\G'$ and tensor product is concatenation.
Further,  we extend $f$ into a map $\Ob(\mathcal{C}')\rightarrow \Z_{\alpha d}$,  such that
  $$f(V,k)=f(V) +dk \quad \text{for}\quad (V,k)\in\G' .$$
Given two elements $(V,k)$ and $(W, l)$ of $\G'$
we define their tensor product as $(V\otimes W, k+l)$.
We  allow in $\C'$ direct sums of objects with homogenous $f$ value. 

 For  $X=(V,k)\in \Ob(\C')$ and $Y=(W,l)\in \Ob(\C')$ we set
% \mbox{$X=((V_1,k_1),\dots,(V_\mu,k_\mu))$} and \mbox{$Y=((W_1,l_1),\dots,(W_\nu,l_\nu))$}, by
$$\mathrm{Hom}_{\mathcal{C}'}(X,Y)=\left\{\begin{array}{l} 
 0 \text{ if }f(X)\neq f(Y)\\
 \mathrm{Hom}_{\mathcal{C}}(V , W)\text{ else.}
 \end{array}\right.$$

The category   ${\mathcal{C}'}$ is a  tensor category over $\k$ with unit object $(\1,0)$ and compatible duality.
Note that $(V^*,l)$ is a left and right dual for $(V,k)$ if $l$ is choosen so that 
 $f(V^*,l)=-f(V,k)$.

 \begin{pro}
 The category ${\mathcal{C}'}$ is semisimple with $\Gamma'$
 as representative set of simple objects. 
 The group of invertible objects is $G'\cong \Z_d\times \Z_\alpha$, generated by $(t,0)$
 and $(\1,1)$.
 \end{pro}
\begin{proof}
%It is easy to see that $\C'$ decomposes into a disjoint union of full subcategories
%$\{\C'_i\}_{i\in\Z_{\a d}}$.
The object $(V,j)$ is invertible in $\mathcal{C}'$
if and only if $V$ is invertible in $\mathcal{C}$. We deduce the last statement.
To prove semisimplicity, it is enough to decompose the tensor product of two objects objects $(V,j)$, $(V',j')$ in $\Gamma '$. We have in the category $\mathcal{C}$ a decomposition
$$\Id_V\otimes \Id_{V'}=\sum_W\sum_{\nu\in I_{V,V'}^W} b_\nu \Id_W a_\nu\ .$$
We set $\chi_{V,V'}^W=(f(V)+f(V')-f(W))/d$. Then we have in the category $\mathcal{C}'$ the following decomposition
$$\Id_{(V,j)}\otimes \Id_{(V',j')}=\sum_W\sum_{\nu\in I_{V,V'}^W} b_\nu \Id_{(W,j+j'+\chi_{V,V'}^W)} a_\nu\ .$$
\end{proof}

%This defines on $\C'$ the structure of $\Z_{\alpha d}$--category. 
Further, let us give $\C'$ a ribbon structure which twists the one given on $\mathcal{C}$.
The braiding is given by a formula
 $$b'_{(V,k),(W,l)}=\xi^{-f(V,k)f(W,l)}b_{V,W} \ ,$$
 with  appropriate choice of a root of unity $\xi$ whose order
 $2\alpha d$ if $d$ is even and $\alpha d$ if $d$ is odd.
Using duality, the twist is then  given by
 $$ \theta'_{(V,k)}=\xi^{-{f(V,k)}^2}\theta_V\ .$$
The corresponding colored link invariants
$F=F_\C$ and $F'=F_{\C'}$ are equal up to a power of $\xi$ which is computed from map $f$ and linking numbers.
Note that $\xi$ is choosen such that $\alpha m$ elements of $G'$ become transparent.

The $G'$--category $\mathcal{C'}$ is premodular and can be modularized as described in  \cite{Bru, BB}. Simple objects in the  modularization 
$\widetilde{\mathcal{C}}$ are obtained from those of $\mathcal{C'}$ quotienting by a free action. The set $\tilde \G$
of simple objects in $\tilde \C$ has cardinality
$|\G|/m$. 
Below we give a detailed proof of the decomposition formula in the spin case, the cohomological cases can be
proven similarly. 

%\subsection*{Invertible objects in  $\mathcal{C'}$} 
% Note that the lifting $f'$ defines a $2$-cocycle on $G$ with value
% in $\Z/\alpha$. More generally, if $\Hom_{\mathcal{C}}(U\otimes V,W)$
% is non trivial, then we get a correction term $c_W(U,V)\in\Z/\alpha$ defined by
% $$\Deg(W)=\Deg(U)+\Deg(V)+c_W(U,V)d\ .$$
 
%Let $\Gamma'=\Gamma\times \Z_\alpha$. The objects of the category ${\mathcal{C}'}$ are finite sequences of elements of $\Gamma'$. The monoidal structure is defined by concatenation.

%The group of characters $\widehat G$ is identified 
%with $\Z_d$ using the root $\kappa$.  We lift the degree 
% map $\deg :\Gamma\rightarrow \Z_d$ into a map $f :\Gamma\rightarrow \Z_{\alpha d}$; this is done by representing $\deg(V)$ by an integer lower 
%than $d$ which is further considered modulo $\alpha d$. 
% We extend this  map $f$ to $\Gamma'$ by the formula
%$$f(V,k)=f(V)+kd\in\Z_{\alpha d}$$
% and then additively to any sequence of those.
% The $\k$-vector space of morphisms $\Hom_{{\mathcal{C}'}}(X,Y)$ between two objects
% \mbox{$X=((V_1,k_1),\dots,(V_\mu,k_\mu))$} and \mbox{$Y=((W_1,l_1),\dots,(W_\nu,l_\nu))$}
% is zero if $f(X)\neq f(Y)$, and is equal to 
% $ \mathrm{Hom}_{\mathcal{C}}(V_1\otimes\dots\otimes V_\mu  , W_1\otimes\dots\otimes W_\nu)$
% else.
 
%We will need the following proposition.

\subsection*{Proof of Theorem 3, spin case}
We consider here the spin case, which means that $d$ is even, $m$ is odd and the twist $\theta_t$ has order $2m$. The generator $t$ can be choosen so that $\theta_t=e^{\frac{i\pi}{m}}$.
Let $\xi=e^{\frac{i\pi l}{\alpha^2 m}}$ with $\beta^2 l\equiv 1+\alpha^2 m$ mod $2\alpha^2 m$.
%  We make ${\mathcal{C'}}$ into a ribbon category with braiding given by the formula
% $$b'_{(V,k),(W,l)}=\xi^{-f(V,k)f(W,l)}b_{V,W} \ .$$
% Using duality, the twist is then  given by
% $$ \theta'_{(V,k)}=\xi^{-{f(V,k)}^2}\theta_V\ .$$
 Note that $\xi^{\alpha d=1}$ so that  the modified braiding is well defined.

 The braiding coefficients for the generators of $G'$ are:
 $$b'_{X,(t,0)}= (\kappa\xi^{-2\delta})^{f(X)}=e^{\frac{2i\pi(1-l\beta^2)f(X)}{ d }}\ ,$$
 $$b'_{X,(\1,1)}= (\xi^{-2d})^{f(X)}=e^{\frac{2i\pi lf(X)\beta}{\alpha}}\ .$$ 
 The twist coefficients are
 $$\theta'_{(t,0)}=e^{\frac{i\pi}{ m}}e^{-\frac{i\pi l\beta^2}{ m}}=-1\ ,$$
 $$\theta'_{(\1,1)}=e^{-\frac{i\pi l d^2}{ \alpha^2 m}}=1\ .$$
 
It follows that the group of transparent objects is generated by $(t,0)^\beta$ which has trivial twist and quantum dimension $1$.
 Applying the results  \cite{Bru, BB} we see that the category $\mathcal{C}'$ is modularizable,
 i.e. that there exists a modular category $\widetilde{\mathcal{C}}$ and a 
 dominant  ribbon functor $\mathcal{C}'\rightarrow \widetilde{\mathcal{C}}$.
 Here the group of transparent objects acts freely on the set $\Gamma'$ of simple objects in $\mathcal{C}'$. 
This is proved using the map $f$ and the fact that $f((t,0)^\beta)=\beta \delta$ has order $\alpha m$ which is the order of $(t,0)^\beta$. 
 Hence, the
 simple objects $\widetilde \G$ in $\widetilde{\mathcal{C}}$ are  represented by cosets in $\G'$ under this free action.

 In the category $\widetilde{\mathcal{C}}$ the group of invertible objects is $\widetilde G=\la(t,0),(\1,1)\ra\cong \Z_\beta\times \Z_\alpha\cong \Z_\delta$. The twist coefficient for $(t,0)$ is equal to $-1$, so
 the category $\widetilde{\mathcal{C}}$ is $\delta$--spin.
 Denote by $\omega'$ the Kirby element in $\mathcal{C}'$ which represents 
 $\alpha m$ times the Kirby element $\widetilde{\omega}$ in $\widetilde{\mathcal{C}}$. We write  the graded decomposition
 $$\omega'=\sum_{c\in \Z_\delta}\omega'_c=\alpha \sum_{c\in \Z_\delta}\omega_c\ ,$$
% where the isomorphism $\widehat{G}'\cong \Z_\d$ is normalized so that the $\delta$-grading is $f$ mod $\d$,
 so that 
 $$\omega'_c=\sum_{\gamma\in \Z_{\a d}\atop \gamma\equiv c \ \text{mod}\ \d}\sum_{X\in \G'\atop f(X)=\gamma} \la X\ra X=\sum_{\gamma\equiv c \ \text{mod}\ \d} \omega'_{|f=\gamma}\ .$$
% We use the notation $$\omega'_{|f=\gamma}=\sum_{X\in \G'\atop f(X)=\gamma\ }\mathrm{qdim}(X) X\ .$$
 Moreover, for any $\Z_{\a d}\ni\gamma\equiv c$ mod $\delta$, the Kirby color $\widetilde{\omega}_c$ in $\widetilde{\C}$ is represented by $\omega'_{|f=\gamma}$, i.e.
 $$\widetilde{\omega}_c=\omega'_{|f=\gamma}=\frac{1}{\alpha m}\;\omega'_c\ .$$
Indeed, the set $\omega'_{|f=\gamma}$ consists of all $(V,k)\in \Gamma'$ such that $f(V)=\gamma-kd$ mod ${\a d}$.
There are $|\Gamma|/d$ such elements.  Acting with $(t^\beta,0)$ we can shift the degree of solutions by $\delta$. In this way
we obtain all $(V,k)\in \Gamma'$ with $\deg(V)= c$ mod $\delta$. Taking the quotient by this action we get $\tilde \omega_c$.

 It makes sense to evaluate both Reshetikhin-Turaev ribbon functors $F=F_\mathcal{C}$, and $F'=F_\mathcal{C'}$
 on $\mathcal{C}'$ colored links. Let $M=S^3(L)$ be a $3$--manifold given by surgery on the $n$-component link $L$ with signature $(b_+,b_-)$, and 
 $\sigma\in \Spin_\delta(M)$ represented by coefficients $c_j\in \Z_\delta$, $1\leq j\leq n$.
% For each $j$, $1\leq j\leq \nu$ we choose $ \gamma_j\in \Z_{d\alpha}$ a lifting of $c_j$.
For objects $X_j\in \G'$, $1\leq j\leq n $, we have
$$F(L(X_1, \dots,X_n))=\xi^{ ^tf(X)Lf(X)} F'(L(X_1,\dots,X_n))\ $$ 
where $^tf(X)Lf(X)=\sum_{i,j}L_{ij}f(X_i)f(X_j)$.
%Here $\lk_{jj'}$ is the linking number if $j\neq j'$ and framing coefficient if $j=j'$.
 Note that the left hand side is invariant under action of $(t,0)^\beta$ on objects and  can be used for the evaluation of the reduced invariant 
$\tau_{\widetilde{\mathcal{C}}}(M,\sigma)$ which we want to compare with $\tau_{{\mathcal{C}}}(M,\sigma)$. We have
%\begin{equation}
%\tau_{\widetilde{\mathcal{C}}}(M,\sigma)=\frac{F'((L_j,\omega'^{ (\gamma_j)})_{1\leq j\leq \nu)})}
% {\widetilde{\Delta}^{b_+}\widetilde{\Delta}^{b_-}}\end{equation} 

\begin{align*}
 \tau_\mathcal{C}(M,\sigma)&=\frac{F(L(\omega_{c_1},\dots,\omega_{c_n}))}
 {{(F(U_1(\omega_{\delta/2})))^{b_+}(F(U_{-1}(\omega_{\delta/2})))^{b_-}}}\\
 &=\frac{\alpha^{-n+b_++b_-}F(L(\omega'_{c_1},\dots,\omega'_{c_n}))}
 {{(F(U_1( \omega'_{\delta/2} )))^{b_+}(F(U_{-1}( \omega'_{\delta/2})))^{b_-}}}
\end{align*}
\begin{align*}
 F(L(\omega'_{c_1},\dots,\omega'_{c_n}))&=\sum_{\gamma\in(\Z_{\a d})^n\atop \gamma_i \equiv c_i\ \mathrm{mod}\ \delta}F(L(\omega'_{|f=\gamma_1},\dots,\omega'_{|f=\gamma_n}))\\
 &=\sum_{\gamma\in(\Z_{\a d})^n\atop\gamma_i \equiv c_i\ \mathrm{mod}\ \delta} \xi ^{^t\gamma L\gamma}
 F'(L(\omega'_{|f=\gamma_1},\dots,\omega'_{|f=\gamma_n}))\\
 &=  \sum_{\gamma\in(\Z_{\a d})^n\atop \gamma_i \equiv c_i\ \mathrm{mod}\ \delta}  \xi ^{^t\gamma L\gamma}F'(L(\tilde\omega_{c_1},\dots,\tilde\omega_{c_n}))
  \\
 &=  F'(L(\tilde\omega_{c_1},\dots,\tilde\omega_{c_n}))
\sum_{\gamma\in(\Z_{\a d})^n\atop \gamma_i \equiv c_i\ \mathrm{mod}\ \delta}  \xi ^{^t\gamma L\gamma}
 \end{align*} 
%\AB{I think would be nice to have a short argument for the last equality}
After normalization we get
\begin{align*}
\tau_\C(M,\sigma)&=\alpha^{-b_1(M)}\tau_{\widetilde{\C}}(M,\sigma) \,g^{-b_+}(\bar{g})^{-b_-}\sum_{\gamma\in(\Z_{\a d})^n \atop \gamma \equiv c\  \mathrm{mod}\ \delta}\xi ^{^t\gamma L\gamma} \quad{\text {where}}\\
 g&=\sum_{ \gamma\in\Z_{\a d}\atop   \gamma\equiv \delta/2\ \text{mod}\ \delta}\xi^{\gamma^2}\ 
\end{align*} 
and $\bar g$ is the complex conjugate.
 One  can check, following the graded construction in Section {\ref{refined}}
% (without the decomposition statement) 
that the formula 
 $$\tau_\xi^\MOO(M,\sigma)=g^{-b_+}(\overline{g})^{-b_-} \sum_{\gamma \equiv c\  \mathrm{mod}\ \delta}\xi ^{^t\gamma L\gamma}\ $$
 defines an invariant of $(M,\sigma)$.
 We conclude
 $$\tau_\C(M,\sigma)=\alpha^{-b_1(M)}\tau_{\widetilde{\C}}(M,\sigma)\tau_\xi^\MOO(M,\sigma)\ . $$
%In the case, $\alpha=1$,
In the case $\gcd(m,\d)=1$, we have $\alpha=1$ and $\sum_{j}L_{ij} \gamma_j=\frac{\delta}{2}\, L_{ii}$ mod $\delta$. Furthermore,
we can assume that our surgery presentation has even linking matrix (the obstruction  given by the spin cobordism group  vanishes),
 so that $c\in \Ker L$ mod $\d$. Decomposing $\gamma=\gamma_0+\delta x$ with  $\gamma_0\in \Ker L$ and $x\in \Z_m$ we see that
$$ \sum_{\gamma \equiv c\  \mathrm{mod}\ \delta}\xi ^{^t\gamma L\gamma}=
\sum_{x\in(\Z_m)^n}\xi ^{^t(\gamma_0+\delta x) L(\gamma_0+\delta x)}=\sum_{x\in(\Z_m)^n}\xi^{^t(\delta x) L (\delta x)}=
\sum_{x\in(\Z_m)^n}\xi^{^tx L x}
$$
does not depend on  $c$. 
%Observe that there is $y\in \Z_m$, such that where $2y=1$ mod $m$, since $m$ is odd.
Summing over $\sigma$ we get
 $$\tau_\C(M)=\tau_{\widetilde{\C}}(M)\tau_\xi^\MOO(M)\  .$$ 
%Note that
%the MOO invariants can be produced from modular categories whose simple objects are all invertible. If the number of simple objects is even, 
%there exists a model with
% generator $t$ of quantum dimension $+1$ and a model with generator $t$ of quantum dimension $-1$.
%Here we used the last model and choose $\eta=\xi^{\delta l}$ the root of unity of order $m$, 
\hfill $\square$\\

 \section*{Appendix}
Here we give a simple direct proof of 
   Theorem 3 in the case when $\gcd(m, \delta)=1$. For  readers convenience, we repeat the statement.

\begin{thm}\label{decomposition} Let $\C$ be
a modular $d$--category with the group $G=\la t\ra$ of invertible objects, 
such that $\deg t=\delta$, $\gcd(\delta, d/\delta)=1$.
Then  there are exists a subcategory $\widetilde \C \subset \C$ and a root of unity $\xi$, such that
for any closed orientable $3$--manifold $M$
$$\tau_\C(M)=\tau_{\widetilde{\C}}(M)\tau^{\rm {MMO}}_{\xi}(M) .$$
\end{thm}

%Let $\C$ be a  modular category  with group of invertible object $\mbox{$G_\mathcal{C}=G=\la t\ra\cong\Z_d$}$
%such that  $\deg(t)=\delta\in \Z_d\cong \widehat{H}$,   $\delta|d$ and $(\delta, d/\delta)=1$.
%In this case we will define a reduced full modular 
%subcategory $\widetilde \C$ such that the Reshetikhin-Turaev invariant of $\C$ decomposes
%as the product of the quantum  invariant of $\widetilde \C$  and an invariant
%derived from the linking matrix associated with a surgery presentation (\cite{MOO}).

Let $m$ be  such that $\delta  m=d$  and 
we set $$\widetilde{\Gamma}:=\{\l\in \Gamma\;\;| \;\; \deg(\l)=0\pmod m\}.$$
Let $\widetilde{\C}$ be the full ribbon subcategory of $\C$ generated by $\widetilde{\Gamma}$ 
and $\widetilde{\omega}$ be the corresponding Kirby color.
Let $K=\langle t^m\rangle.$ In this situation $K$ has order $\delta$
and $\widetilde{\Gamma}=\Gamma_0\amalg \Gamma_m\amalg \cdots \amalg \Gamma_{(\delta-1)m}.$
Moreover, consider $\eta \in \mathbb{C}$ such that 
\begin{center}
\psfrag{a}{$t^{\delta}$}
\psfrag{b}{$t^{\delta}$}
\psfrag{c}{$=\eta$}
\includegraphics{transparent2.eps}
\end{center}
%In particular, we have $\eta^2=e_d^{\delta^3}$.

%From Lemma \ref{lem21} it is easy to deduce the next Proposition.

We will need the following proposition.
\begin{pro} \label{crossing} Let $\xi= \eta  \langle t \rangle ^{\delta}$.
Then we have the following equalities of morphisms:
 \begin{enumerate}[(a)]
  \item for any $\l\in \widetilde{\Gamma}$ and $k\in \{0,\cdots,m-1\}$
\begin{center}
\psfrag{b}{$t^{k\delta}$}
\psfrag{a}{$\l$}
\psfrag{c}{$\;=\;\;$}
\includegraphics{transparent.eps}
\end{center}

\item for any $k,s\in \{0,\cdots,m-1\}$
\begin{center}
\psfrag{a}{$t^{k\delta}$}
\psfrag{b}{$t^{s\delta}$}
\psfrag{c}{$\;=\; \eta^{ks}\;$}
\includegraphics{transparent2.eps}
\end{center}
\item for any $k,s\in \{0,\cdots,m-1\}$
\begin{center}
\psfrag{a}{$t^{k\delta}$}
\psfrag{b}{$t^{s\delta}$}
\psfrag{c}{$\;=\; \xi^{2ks}\;$}
\includegraphics{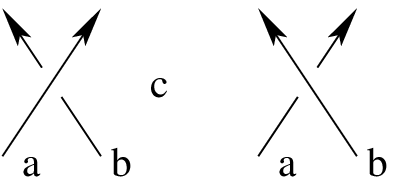}
\end{center}

\item for any $k\in \{0,\cdots,m-1\}$, the twist coefficient $\theta_{t^{k\delta}}=\xi^{k^2}$.
%and $\la t\ra =1$.

\end{enumerate}
\end{pro}
The proof is a straightforward computation using that $\langle t\rangle=\pm 1$ and 
$\xi^2=\eta^2=e_d^{\delta^3}$.

\begin{lem}
The category $\widetilde{\C}$ is modular and reduced, in the sense that it is ${\widetilde{G}}$--refinable.
\end{lem}

\begin{proof}
Clearly,
 $1\in \widetilde{\G}$ since $\deg(1)=0$.
Duality axiom holds since for $\l\in \widetilde{\G}$, $\deg(\l^*)=-\deg(\l)$
and hence $\l^*\in \widetilde{\G}$. Domination follows trivially from the same property of $\C$.
It remains to prove  the non-degeneracy.

Observe that  the Kirby color $\omega$ decomposes as the sum
\begin{equation*}
\omega=\sum_{i=0}^{m-1}  \langle t \rangle^{i\delta } t^{i\delta }\widetilde{\omega}
\end{equation*}
since $(\d,m)=1$.
%Indeed, since $(\delta^2,m)=1$, we have 
%$\{i\delta^2+jm \pmod d|i=0,\cdots,m-1, j=0,\cdots, \delta-1\}=\{0,\cdots, d-1\}$
%and  we can write the Kirby color as follows:
%\begin{equation*}
%\omega=\sum_{i=0}^{m-1} \sum_{j=0}^{\delta-1}t^{i\delta} \omega_{jm}
% =\sum_{i=0}^{m-1} t^{i\delta} \widetilde{\omega}\, .
%\end{equation*}
%Note that $\la t\ra=1$ by Lemma \ref{lem1}.
Decomposing $\omega$ as above and using Proposition \ref{crossing} (a) we get  for any $\lambda \in \widetilde{\Gamma}$
the following equalities 
\begin{align*}
 \psdiag{10}{20}{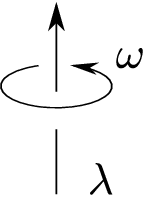}&=\sum_{i=0}^{m-1} \langle t \rangle^{i\delta } \psdiag{10}{20}{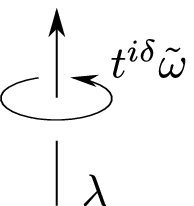}=\sum_{i=0}^{m-1}
 \langle t \rangle^{i\delta } \psdiag{10}{22}{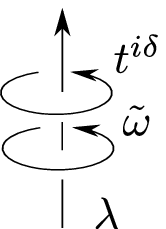}\\
 &=\sum_{i=0}^{m-1}\langle t \rangle^{2i\delta } \psdiag{10}{20}{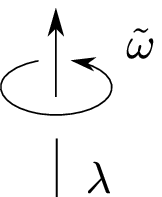}=m \psdiag{10}{20}{decomposable04}.
\end{align*}
The \emph{killing property} combined with Proposition \ref{twisted} implies non-degeneracy.
\end{proof}

%\begin{lem} 

%\end{lem}

%\AB{In the definition of the modular category we do not use a pair, so you can write
%$\C$ with the set $\G$, but not the pair, please correct everywhere}

%\begin{proof} As constructed, the category $\widetilde \C$ has $\widetilde \G$ as a representative  
%set of simple objects and its group of invertible objects is \mbox{$G \cap\widetilde \G =\widetilde{G}$} in trivial
%$\widehat {\widetilde{G}}$-grading.
%The set $\widetilde \G_0$ of degree zero elements of $\widetilde \G$ with respect 
%to the $\widetilde G$-grading consists of all $\widetilde \lambda \in \widetilde \G$
%satisfying 
%$$\psdiag{10}{21}{grading11}= \psdiag{10}{20}{grading21},$$
%hence $\widetilde \G_0=\widetilde \G$. Since $H\subset \widetilde \G_0$ and 
%$H\subset G$, it follows that $H$ is a degree zero subgroup of
%$\widetilde G$ with respect to the $\widetilde G$-grading, so the category
%$\widetilde \C$ is $H$--refinable.

%\end{proof}

%With $\C$, $\widetilde{\C}$ and $\xi$ as above, we have the following decomposition formula for the
%Reshetikin-Turaev invariants:

\subsection*{Proof of Theorem }
Let $M$ be presented by 
an oriented framed link
$L=(L_1,\cdots, L_n)$ with all components $L_i$ unknotted (such a link always exists, see \cite{Lickorish}).
Then, the invariant
\begin{equation*}
\tau_\C(M)=\frac{F_\C(L(\omega,\cdots,\omega))}{(F_\C(U_1(\omega)))^{b_+}(F_\C(U_{-1}(\omega)))^{b_-}}.
\end{equation*}
If we replace $\omega=\sum_{i=0}^{m-1}\langle t \rangle^{i\delta } t^{i\delta}\widetilde{\omega}$ in $F_{\C}(L(\omega,\cdots,\omega))$
and apply Proposition (a) we obtain:
\begin{equation*}\label{first}
F_{\C}(L(\omega,\cdots,\omega))=F_{\widetilde{\C}}(L(\widetilde{\omega},\cdots,\widetilde{\omega}))\sum_{i_1,\cdots,i_n=0}^{m-1}
\langle t \rangle^{(i_1+\cdots +i_n)\delta } F_{\C}(L(t^{i_1\delta},\cdots,t^{i_n\delta})).
\end{equation*}
In particular, for $\epsilon=\pm 1$, we have
\begin{equation*}
F_{\C}(U_\epsilon(\omega))=F_{\widetilde{\C}}(U_\epsilon(\widetilde{\omega}))\sum_{i=0}^{m-1} \langle t \rangle^{i\delta } F_\C(U_\epsilon(t^{i\delta})).
\end{equation*}

Given the link $L$ with components colored by $t^{i_1\delta},\cdots, t^{i_n\delta}$, 
$F_{\C}(L(t^{i_1\delta},\cdots,t^{i_n\delta}))$ can be computed as follows: first 
we make each component of $L$ zero framed,
then we unlink the components (using Proposition  (c) and (d) as many times as necessary).
Finally, we obtain $n$ disjoint and unlinked copies of the zero framed unknot with colors $t^{i_1\delta},\cdots, t^{i_n\delta}$
and the relation: 
\begin{equation*}
F_{\C}(L(t^{i_1\delta},\cdots,t^{i_n\delta}))=\langle t \rangle^{(i_1+\cdots +i_n)\delta }\cdot
 \xi^{(i_1,\cdots,i_n) L (i_1,\cdots,i_n)^t },
\end{equation*}
where $(L_{ij})$ is the linking matrix of $L$. Similarly
\begin{equation*}
 F_\C(U_\epsilon(t^{i\delta}))= \langle t\rangle^{i\delta}\cdot \xi^{\epsilon i^2}
\end{equation*}
and the Reshetikin-Turaev invariant decomposes as
\begin{gather*}
\tau_\C(M)=\tau_{\widetilde{\C}}(M) \frac{\sum_{l\in (\Z_m)^n}\xi^{^tl L l}}{(\sum_{i\in \Z_m} \xi^{i^2})^{b_+}
(\sum_{i\in \Z_m} \xi^{-i^2})^{b_-}}.
\end{gather*}

Note that for $d$ even, $\xi$ is an $m$th root of unity if $m$ is odd
and a $2m$th root of unity if $m$ is even and according to \cite{MOO}, 
\begin{equation*}
\tau^{\rm MMO}_\xi(M)=\frac{\sum_{l\in (\Z_m)^n}\xi^{^tl L l}}{(\sum_{i\in \Z_m} \xi^{i^2})^{b_+}
(\sum_{i\in \Z_m} \xi^{-i^2})^{b_-}}
\end{equation*}
is a topological invariant of $M$ independent on the choice of $L$, known as the Murakami--Ohtsuki--Okada
invariant. \hfill$\square$

%\begin{note}

%\end{note}

%\end{proof}
%

\end{document}